\documentclass[final,leqno]{siamltex704}

\usepackage{amsmath,amssymb}% for \begin{bmatrix}, \begin{smallmatrix}, \begin{align} etc.
\usepackage{amsfonts}% for \mathbb{} etc. `blackboard' for uppercase only.
\usepackage{exscale}
\usepackage{graphicx}
\usepackage[pdftex]{thumbpdf}      %%% thumbnails for pdflatex. Run thumbpdf at the end
\usepackage[pdftex,                %%% hyper-references for pdflatex
pdfstartview=FitH,%
bookmarks=true,%                   %%% generate bookmarks ...
bookmarksnumbered=true,%           %%% ... with numbers
bookmarksopen=true,%
hypertexnames=false,%              %%% needed for correct links to figures !!!
breaklinks=true,%                  %%% break links if exceeding a single line
  colorlinks=true,%
  linkcolor=blue,anchorcolor=blue,%
  citecolor=blue,filecolor=blue,%
  menucolor=blue]%
{hyperref} 
\usepackage{etoolbox}  % if then statement to get longer version for tech. report
\newtoggle{report}
\toggletrue{report} % set longer tech report
\makeatletter
\def\@cite#1#2{{\rm [}{{\rm#1}\if@tempswa , #2\fi}{\rm ]}}
\makeatother

\newtheorem{algorithm}{Algorithm}[section]

\title{Absolute value preconditioning for symmetric indefinite linear systems
\thanks{%
the current version generated \today.
This material is based upon work partially supported by the National Science Foundation under Grant No.~1115734. 
The work is partially based on the PhD
thesis of the first coauthor~\cite{thesis}.
}}

\author{%
Eugene Vecharynski\footnotemark[2] 
\and
Andrew V. Knyazev\footnotemark[3]\ \footnotemark[4]\ \footnotemark[5]
}

\begin{document}
%siam_id=72574
%CODEN=SJMAEL
%\slugger{sisc}{?}{?}{?}{?--?}
\maketitle

\setcounter{page}{1}

\renewcommand{\thefootnote}{\fnsymbol{footnote}}
\footnotetext[2]{
%Department of Computer Science and
%Engineering, University of Minnesota, 200 Union Street S.E., Minneapolis, MN 55455
%(eugenev[at]cs.umn.edu).
Computational Research Division, Lawrence Berkeley National Laboratory, Berkeley, CA 94720
(eugene.vecharynski@gmail.com)
}
\footnotetext[3]{Department of Mathematical and Statistical Sciences; 
University of Colorado Denver,
P.O. Box 173364, Campus Box 170, Denver, CO 80217-3364, USA
(andrew.knyazev[at]ucdenver.edu)}
\footnotetext[4]{Mitsubishi Electric Research Laboratories; 201 Broadway
Cambridge, MA 02139}
\footnotetext[5]{
\url{http://www.merl.com/people/?user=knyazev} and \url{http://math.ucdenver.edu/~aknyazev/}
%\url{http://math.ucdenver.edu/~aknyazev/}
}
\renewcommand{\thefootnote}{\arabic{footnote}}

\begin{abstract}
We introduce a novel strategy for constructing symmetric positive
definite (SPD) preconditioners for linear systems with symmetric
indefinite matrices. 
The strategy,  called absolute value preconditioning, is motivated by the
observation that 
the preconditioned minimal residual method
with the inverse of
the absolute value of the matrix as a preconditioner
converges to the exact solution of the system in at most two steps.
Neither the exact absolute value of the matrix nor its
exact inverse are computationally feasible to construct in general.
However, 
we provide a practical example of
an SPD preconditioner that is based on the suggested approach.
In this example we consider a
model problem with a shifted discrete negative Laplacian, and suggest a
geometric multigrid (MG) preconditioner, where the inverse of the matrix absolute
value appears only on the coarse grid, while
operations on finer grids are based on the Laplacian.
Our numerical tests demonstrate
practical effectiveness of the new MG preconditioner, which leads to a 
robust iterative scheme with minimalist memory requirements.
% for
%moderately small shifts.
\end{abstract}


% 15A06 -- Linear equations 
% 65F08 -- Preconditioners for iterative methods
% 65F10 -- Iterative methods for linear systems
% 65N22 -- Solution of discretized equations
% 65N55 -- Multigrid methods; domain decomposition

%\begin{DOI}
% 
%\end{DOI}

\pagestyle{myheadings}
\thispagestyle{plain}
\markboth{EUGENE VECHARYNSKI AND ANDREW V. KNYAZEV}{ABSOLUTE VALUE PRECONDITIONING} %50 Characters Limit

\section{Introduction}\label{sec:intro}

%For a relatively small problem size $n$, the \textit{exact} (up to the effects of round-off errors) 
%solution of~(\ref{eqn:sys}) can be
%efficiently found using a \textit{direct} method; see, e.g., the survey in~\cite{Bunch.Parlett:71}. 
%However, in most cases, the computational cost of such a method does not \textit{optimally} scale with respect
%to $n \rightarrow \infty$, i.e., the amount of involved computations does not 
%grow \textit{proportionally} to 
%%linearly depend on
%the increasing number of matrix elements.
%Therefore, if the problem size is significantly large,
%the application of a direct method
%may become infeasible. 
%Additionally, it is often required to find
%only an \textit{approximate} solution of~(\ref{eqn:sys}), as opposed
%to the \textit{exact} solution targeted by a direct solver.   

%The above arguments motivate an \textit{iterative}
%technique, which
%can be expected to provide a (nearly) optimal complexity and, 
%instead of the exact solution of~(\ref{eqn:sys}), 
%%allows to find 
%delivers a sequence 
%of approximations. 
%%In this manuscript, it is
%%assumed 
%Here we assume
%that $A$  
%is extremely large and possibly sparse. 
%We focus only on \textit{iterative} methods for solving linear system~(\ref{eqn:sys}).
%Let us remark that, in the suggested framework,
%under ``solving a linear system'' we understand ``\textit{approximately} solving a linear system,'' i.e., 
%finding a satisfactory approximation
%to the exact solution of~(\ref{eqn:sys}).
%
%

Large, sparse, symmetric, and indefinite systems 
arise in a variety of applications.
For example, in the form of saddle point problems, such systems result from mixed finite element
discretizations of underlying differential equations of fluid and solid mechanics; see, e.g.,~\cite{Benzi.Golub.Liesen:05}
and references therein. 
In acoustics,
large sparse symmetric indefinite systems are obtained after
discretizing the Helmholtz equation for certain media types and boundary conditions. 
%; see, e.g.,~\cite{Tikhonov.Samarskii:77}.
Often the need to solve symmetric indefinite problems comes as an auxiliary
task within other computational routines, 
such as the  
inner 
%Newton 
step in interior point methods in linear and nonlinear 
optimization~\cite{Benzi.Golub.Liesen:05, Nocedal.Wright:99}, 
or solution of the correction equation in 
the Jacobi-Davidson method~\cite{Sleijpen.Vorst:96} for a symmetric eigenvalue problem.  

%We consider a system of linear equations $Ax = b$,
We consider an iterative solution of a linear system $Ax = b$,
where the matrix $A$ is real nonsingular and symmetric indefinite, 
i.e., the spectrum of $A$ contains both positive and negative eigenvalues.
%\iftoggle{report}{
%There is a number of iterative methods developed \textit{specifically} to solve
%such systems, ranging from modifications of Richardson's iteration, 
%e.g.,~\cite{Calvetti.Reichel:96,Lebedev:69,Saad:83}, to the optimal Krylov subspace minimal residual method
%delivered through short-term recurrent schemes, 
%such as the MINRES algorithm~\cite{Paige.Saunders:75}. 
%In practical problems, the matrix $A$
%can be very ill-conditioned. Along with the location of the spectrum %of $A$ 
%on both sides of the origin, this makes the straightforward application 
%of existing techniques inefficient because of extremely slow convergence.
%} 
%
In order to improve the convergence, we introduce a \textit{preconditioner} $T$ and formally
replace $Ax = b$ by
%\begin{equation}\label{eqn:prec_sys}
the \textit{preconditioned system} $TAx  = Tb$.
%\end{equation} 
If $T$ is properly chosen, an iterative method for
%~(\ref{eqn:prec_sys}) 
this system
can exhibit
a better convergence behavior compared to 
a
%the original 
scheme applied to $Ax = b$.  
Neither the preconditioner $T$ 
%itself 
nor the \textit{preconditioned matrix} $TA$ is normally explicitly computed.

If $T$ is \textit{not} symmetric positive definite~(SPD), then $TA$, in general, 
is not symmetric with respect to any inner
product~\cite[Theorem~15.2.1]{Parlett:98}. 
Thus, the introduction of a non-SPD preconditioner replaces the 
original \textit{symmetric} problem $Ax=b$ by 
a generally \textit{nonsymmetric} $TAx=Tb$.
Specialized methods for \textit{symmetric} 
linear systems are
no longer applicable to the preconditioned problem, 
and must be replaced by
%which suggests that, instead, 
iterative schemes for \textit{nonsymmetric} linear systems;
%should be employed, 
e.g.,
GMRES or GMRES($m$)~\cite{Saad.Schultz:86}, Bi-CGSTAB~\cite{Vorst:92}, and QMR~\cite{Freund.Nachtigal:91}.

The approach
based on the choice of a non-SPD preconditioner, which leads to solving a nonsymmetric
problem, has several disadvantages.
First, no short-term recurrent scheme that delivers 
an \textit{optimal} Krylov subspace method 
is typically available for
a nonsymmetric linear system~\cite{Faber.Manteuffel:84}.
In practice, this means that implementations of the optimal methods (e.g., GMRES) 
require an increasing amount of work and storage at every new step, and hence 
are often computationally expensive.  
%often resulting in an unsatisfactory performance of the corresponding solver. 
%First, in order to maintain the optimality of a Krylov subspace method, one 
%has to allow the increase of the computational work at every new iteration, which can become prohibitive
%for large problems. 

Second, the convergence behavior of iterative
methods for nonsymmetric linear systems 
%such methods
is not completely understood.
In particular, the convergence %of these methods 
may not be characterized in terms of reasonably accessible 
quantities, such as the 
%(estimated) 
spectrum of the preconditioned matrix; see 
the corresponding results for GMRES and GMRES($m$) in~\cite{Greenbaum.Ptak.Strakos:96,Ve.La:10}.
This makes it difficult to predict computational costs.

If $T$ is chosen to be SPD, i.e., $T = T^* > 0$, then the matrix $TA$
of the preconditioned linear system is symmetric with respect to the $T^{-1}$--inner product
defined by $(u,v)_{T^{-1}} = (u, T^{-1} v)$ for any pair of vectors $u$ and $v$.
% \in \mathbb{R}^n$. 
Here $(\cdot,\cdot)$ denotes the Euclidean inner product $(u,v) = v^* u$, in which the matrices $A$ and $T$ are symmetric. 
%In particular, 
Due to this symmetry preservation, system $TAx = Tb$ can be solved using an \textit{optimal} Krylov subspace method
that admits a \textit{short-term recurrent} implementation, such as preconditioned MINRES (PMINRES)~\cite{Elman.Silvester.Wathen:05, Paige.Saunders:75}. 
Moreover, the convergence of the method can be
fully estimated in terms of the spectrum of $TA$.
%, e.g., preconditioned MINRES \cite{Paige.Saunders:75}, or PMINRES, 
%with the convergence behavior fully described in terms of the (estimates of) spectrum of the preconditioned matrix $TA$.

%Therefore, 
In light of the above discussion,
the choice of an SPD preconditioner for a symmetric indefinite linear system 
can be regarded as natural and favorable, especially if corresponding non-SPD preconditioning strategies fail
to provide convergence in a small number of iterations. 
We advocate the use of SPD preconditioning.

The question of constructing SPD preconditioners for symmetric indefinite
systems has been widely studied in many applications.
For saddle point problems, the block-diagonal SPD preconditioning has been addressed, e.g.,
in~\cite{Fischer.Ramage.Silvester.Wathen:98, Silvester.Wathen:94, Wathen.Silvester:93}.
%We are aware of only a few works that discuss SPD peconditioning
%for symmetric indefinite systems.  
%
In~\cite{Bayliss.Goldstein.Turkel:83}, it was proposed
to use an inverse of the negative Laplacian
as an SPD preconditioner for indefinite Helmholtz problems. 
This approach was further extended in~\cite{Laird:Giles:02}
by introducing a shift into the preconditioner. 
Another strategy was suggested in~\cite{Gill.Murray.Ponceleon.Saunders:92},
primarily in the context of linear systems arising in optimization. It is based on  
the so-called \textit{Bunch-Parlett factorization}~\cite{Bunch.Parlett:71}.

We introduce here a different idea of constructing SPD preconditioners
that resemble the inverse of the absolute value of the coefficient matrix.    
Throughout, the absolute value of $A$ is defined as a matrix function
$\left| A \right|  = V \left| \Lambda \right| V^*$, where $A = V \Lambda V^*$
is the eigenvalue decomposition of $A$. 
%and $\left| \Lambda \right| = \mbox{diag}\{\left|\lambda_j\right|\}$.
% 
%%$A \in \mathbb{R}^{n \times n}$ be a symmetric matrix with eigendecomposition $A = V \Lambda V^*$, where $V$ is an
%$A$ have an eigendecomposition $A = V \Lambda V^*$, where $V$ is an
%orthogonal matrix of eigenvectors, and $\Lambda = \mbox{diag}\{\lambda_j\}$, $j = 1, \ldots, n$, is 
%a diagonal matrix of eigenvalues of $A$. We consider the factorization
%\begin{equation}\label{eqn:absval}
%A = \left| A \right| \mbox{sign}(A) = \mbox{sign}(A) \left| A \right|,
%\end{equation}
%where $\left| A \right|  = V \left| \Lambda \right| V^*$ is a matrix absolute value of $A$, %(matrix absolute value),
%and $\mbox{sign}(A) =  V  \mbox{sign}(\Lambda) V^*$ is a matrix sign of $A$. % (matrix sign);
%%$\left| \Lambda \right| = \mbox{diag}\{\left|\lambda_j\right|\}$, 
%%$\mbox{sign}(\Lambda) = \mbox{diag}\{\mbox{sign}(\lambda_j)\}$. 
%Factorization (\ref{eqn:absval}) is, in fact, a \textit{polar decomposition}, see, e.g., \cite{Horn.Johnson:90}, of 
%the symmetric matrix $A$, 
%with the symmetric positive (semi) definite factor $\left| A \right|$ and the orthogonal factor $\mbox{sign}(A)$.
%If, additionally, $A$ is nonsingular, then $\left| A \right|$ is SPD.
%
We are motivated by the observation that 
%a preconditioned iterative method 
PMINRES 
% minimal residual method
with 
%the inverse of the absolute value of the matrix 
$|A|^{-1}$ as a preconditioner
converges to the exact solution in at most two steps.
We refer to the new approach as the \textit{absolute value} (AV) preconditioning and call the 
corresponding preconditioners the AV preconditioners.

The direct approach for constructing an AV preconditioner 
is to approximately solve $\left|A\right| z = r$.
However, 
$\left|A \right|$ is generally not available,
which makes the application of standard 
% matrix preconditioning 
techniques, such as, e.g., incomplete factorizations, approximate inverses, 
%(see~\cite{Benzi:02}) 
problematic. 
The vector $\left|A\right|^{-1} r$ can also be found using matrix function computations,
normally fulfilled by a Krylov subspace method~\cite{Golub.VanLoan:96,Higham:08}
or a polynomial approximation~\cite{Powell:81, Rivlin:81}. 
Our numerical experience shows that the convergence, with respect to the outer iterations, 
of a linear solver can be significantly improved with this approach, but
the computational costs of approximating $f(A) r = \left|A\right|^{-1} r$ may be too high,
i.e., much higher than 
%We recall that the costs of the construction and application of $T$ should preferably be similar to 
the cost of matrix-vector multiplication with $A$.

%Another approach for constructing an AV preconditioner is to 
%adapt standard matrix preconditioning techniques such as 
%%e.g.,\ strategies based on 
%incomplete factorizations, approximate inverses, etc.; see the survey in~\cite{Benzi:02}. 
%The relation between $A$ and $|A|$, which may be useful in this respect, is 
%$|A| = A -  2V_p \Lambda_p V_p^*,$ 
%where $V_p$ is the matrix of eigenvectors of $A$ corresponding to
%all $p$ negative eigenvalues and $\Lambda_p = \mbox{diag}\left\{ \lambda_1, \ldots, \lambda_p \right\}$. 

Introduction of the \textit{general concept} of the AV preconditioning is the main theoretical
contribution of the present work. As a proof of concept example of the AV preconditioning, 
we use a geometric multigrid (MG) framework. To investigate
applicability and practical effectiveness of the proposed idea, 
%for a particular computational problem, 
we choose a model problem resulting from discretization of
a shifted Laplacian (Helmholtz operator) on a unit square with Dirichlet boundary conditions.
The obtained linear system is real symmetric indefinite. 
We construct an MG AV preconditioner that, used in the 
PMINRES iteration, delivers an efficient computational scheme.   

Let us remark that the same model problem has been considered in~\cite{Bramble.Leyk.Pasciak:93},
%In this work, 
where the authors utilize the coarse grid approximation to reduce the indefinite problem to the SPD system.
%which can be solved by the conjugate gradient method. 
Satisfactory results have been reported 
for small shifts, i.e., for slightly indefinite systems. 
However, the limitation of the approach lies in the requirement on
the size of the coarse space, which should be chosen sufficiently large.  
As we show below, the MG AV preconditioner presented in this paper allows keeping
the coarsest problem reasonably small, even if the shift is large. 
%This means that 
%the cost of the new AV preconditioner is not affected much by the coarsest grid 
%operations. 

Numerical solution of Helmholtz problems is an object of active research; see, 
e.g.,~\cite{Airaksinen.Heikkola.Pennanen:07, Bollhoefer.Grote.Schenk:09, Erlangga.Vuik.Ooster:04,Elman.Ernst.OLeary:01, 
Haber.MacLachlan:11, Osei.Saad:10, vanGijzen.Erlangga.Vuik:07}.
A typical Helmholtz problem is approximated by a complex symmetric (non-Hermitian) system. 
%The real symmetric case of the Helmholtz equation, considered in this paper, is uncommon, in practice.  
The real symmetric case of the Helmholtz equation, considered in this paper, is less common. 
However, methods for complex problems are evidently applicable to our particular real case, which allows us 
to make numerical comparisons with known Helmholtz solvers. 

We test several of solvers, based on the inverted Laplacian and the 
standard MG preconditioning, to compare with the proposed AV preconditioning.    
In fact, the inverted (shifted) Laplacian  preconditioning~\cite{Bayliss.Goldstein.Turkel:83, Laird:Giles:02}
for real Helmholtz problems can be viewed as a special case of our AV
preconditioning. 
In contrast to preconditioners in~\cite{Gill.Murray.Ponceleon.Saunders:92} relying
on the Bunch-Parlett factorization, we show that the AV preconditioners
can be constructed without any decompositions of the matrix, 
which is crucial for very large or matrix-free problems.

This paper is organized as follows.
In Section~\ref{sec:absval_prec}, we present and justify the general notion of an
AV preconditioner.
The rest of the paper deals with the question of whether 
AV preconditioners can be
efficiently constructed in practice.
In Section~\ref{subsec:prec_constr}, we give a positive answer by constructing an example of 
a geometric MG AV
preconditioner for the model problem.
% problem resulting from discretization of
%a shifted negative Laplace operator on a unit square. 
%
The efficiency of this preconditioner is demonstrated in our numerical tests
in Section~\ref{sec:numeric}.
We conclude in Section~\ref{sec:conl}.
%In particular, we show that for highly indefinite systems, 
%MINRES with the new MG AV preconditioner outperforms 
%conventional schemes based on %restarted GMRES with the 
%indefinite MG preconditioning.  
%%We report mesh-independent convergence behavior for PMINRES 
%%preconditioned with the new scheme. 

%Throughout, we assume exact arithmetic. 
%The choice of real vector spaces
%has been made to simplify the presentation.
%Generalization of the results to the complex case is straightforward.
%We note that the results presented in this work are partially based on the PhD
%thesis of the first coauthor~\cite{thesis}, defended at the University of Colorado Denver
%under the supervision of the second coauthor. 

\section{AV preconditioning for symmetric indefinite systems}\label{sec:absval_prec}

Given an SPD preconditioner $T$, 
we consider
solving a linear system with 
the \textit{preconditioned minimal residual method}, implemented in the form of the
preconditioned MINRES (PMINRES) algorithm~\cite{Elman.Silvester.Wathen:05, Paige.Saunders:75}.
In the absence of round-off errors, at step $i$, the method   
constructs an approximation 
$x^{(i)}$ to the solution of $Ax = b$ of the form
\begin{equation}\label{eqn:xkrylov}
x^{(i)}  \in  x^{(0)} + \mathcal{K}_i \left(T A,T r^{(0)}\right), 
%\ \mathcal{K}_i \left(TA,Tr^{(0)}\right)  =  \mbox{span}\left\{Tr^{(0)}, (TA)Tr^{(0)},\ldots,(TA)^{i-1}Tr^{(0)}\right\},
\end{equation}
such that the residual vector $r^{(i)} = b - A x^{(i)}$ satisfies the optimality condition
\begin{equation}\label{eqn:rkrylov}
\|r^{(i)}\|_T = \min_{u \in A \mathcal{K}_i \left(TA,Tr^{(0)}\right)} \| r^{(0)} - u\|_T.
\end{equation}
Here,  
$\mathcal{K}_i \left(TA,Tr^{(0)}\right) = \mbox{span}\left\{Tr^{(0)}, (TA)Tr^{(0)},\ldots,(TA)^{i-1}Tr^{(0)}\right\}$
is the %$i$-dimensional 
Krylov subspace generated by the matrix $TA$ and the vector $Tr^{(0)}$,  
%$A \mathcal{K}_i \left(TA,Tr^{(0)}\right) = \mbox{span}\left\{(AT)r^{(0)}, \ldots,(AT)^{i}r^{(0)}\right\}$
%is the corresponding Krylov residual subspace; 
the $T$-norm is defined by $\|v\|^2_T = (v,v)_T$ for any $v$, and $x^{(0)}$ is the
initial guess.
Scheme~(\ref{eqn:xkrylov})--(\ref{eqn:rkrylov}) represents an \textit{optimal} Krylov subspace method
and the PMINRES implementation is based on a \textit{short-term recurrence}.   
%The PMINRES implementation of the optimal preconditioned minimal residual method~(\ref{eqn:xkrylov})--(\ref{eqn:rkrylov}) 
%is based on a \textit{short-term recurrence}. 
%It can be viewed as the MINRES algorithm applied to~(\ref{eqn:prec_sys})
%with Euclidean inner products replaced by $T^{-1}$--based inner products,
%written in a way to avoid computations involving $T^{-1}$.
The conventional convergence rate bound for~(\ref{eqn:xkrylov})--(\ref{eqn:rkrylov}) can be found, e.g., 
in~\cite{Elman.Silvester.Wathen:05}, and relies solely on the distribution of eigenvalues of $TA$.

%Let 
%%$A \in \mathbb{R}^{n \times n}$ be a symmetric matrix with eigendecomposition $A = V \Lambda V^*$, where $V$ is an
%$A$ have an eigendecomposition $A = V \Lambda V^*$, where $V$ is an
%orthogonal matrix of eigenvectors, and $\Lambda = \mbox{diag}\{\lambda_j\}$, $j = 1, \ldots, n$, is 
%a diagonal matrix of eigenvalues of $A$. We consider the factorization
%\begin{equation}\label{eqn:absval}
%A = \left| A \right| \mbox{sign}(A) = \mbox{sign}(A) \left| A \right|,
%\end{equation}
%where $\left| A \right|  = V \left| \Lambda \right| V^*$ is a matrix absolute value of $A$, %(matrix absolute value),
%and $\mbox{sign}(A) =  V  \mbox{sign}(\Lambda) V^*$ is a matrix sign of $A$. % (matrix sign);
%%$\left| \Lambda \right| = \mbox{diag}\{\left|\lambda_j\right|\}$, 
%%$\mbox{sign}(\Lambda) = \mbox{diag}\{\mbox{sign}(\lambda_j)\}$. 
%Factorization (\ref{eqn:absval}) is, in fact, a \textit{polar decomposition}, see, e.g., \cite{Horn.Johnson:90}, of 
%the symmetric matrix $A$, 
%with the symmetric positive (semi) definite factor $\left| A \right|$ and the orthogonal factor $\mbox{sign}(A)$.
%If, additionally, $A$ is nonsingular, then $\left| A \right|$ is SPD.

The following trivial, but important, theorem regards 
%the inverted absolute value of $A$, i.e., 
$\left| A \right|^{-1}$
as an SPD preconditioner for a symmetric indefinite system. %~(\ref{eqn:sys}).
\begin{theorem}\label{thm:opt_prec} 
The preconditioned minimal residual method~(\ref{eqn:xkrylov})--(\ref{eqn:rkrylov}) 
with preconditioner $T  = \left|A\right|^{-1}$
converges to the 
%exact 
solution of $Ax = b$ in at most two steps. 
\end{theorem}
%\begin{proof}
%Minimization property (\ref{eqn:rkrylov}) at a step $i$ 
%of a preconditioned minimal residual method can be equivalently written as
%\begin{equation}\label{eqn:mr_poly}
%\|r^{(i)}\|_T = \min_{p \in \mathcal{P}_i, \ p(0)=1} \|p(AT) r^{(0)}\|_T,
%\end{equation}
%where $\mathcal{P}_i$ is the set of all polynomials of degree at most $i$.
%Then, according to decomposition~(\ref{eqn:absval}), the choice $T =  \left|A\right|^{-1}$ 
%results in the matrix $AT = \mbox{sign}(A)$ with only two distinct eigenvalues: 
%$-1$ and $1$. Hence, the minimal polynomial of $AT$ is of the second degree. Thus, by~(\ref{eqn:mr_poly}), 
%$\|r^{(i)}\|_T = 0$ for at most $i = 2$. 
%\qquad\end{proof}

Theorem~\ref{thm:opt_prec} implies that $T  = \left|A\right|^{-1}$
is an \textit{ideal SPD preconditioner}. 
%%
%A preconditioner $T$ can be considered ideal if it delivers the 
%preconditioned matrix $TA$ with the corresponding minimal polynomial of the
%least possible degree; e.g., see related discussions in~\cite{Gill.Murray.Ponceleon.Saunders:92, Murphy.Golub.Wathen:99}. 
%%
%The degree of the minimal polynomial 
%gives the number of iterations typically
%required by a Krylov subspace method to guarantee convergence.  
%
%
%Since the symmetric matrix $A$ has both positive and negative eigenvalues, 
%so does $TA$, provided that $T$ is SPD.
%%
%This implies that
%the degree of the minimal polynomial of $TA$ is at least $2$.
%% 
%Therefore, an ideal SPD preconditioner, in general, 
%ensures convergence to the  
%solution of $Ax = b$ in two steps. 
%
%One-step convergence can occur for special choices of the initial guess.     
%We also 
Note that
the theorem 
%Theorem~\ref{thm:opt_prec} 
holds not only for the preconditioned minimal residual 
method~(\ref{eqn:xkrylov})--(\ref{eqn:rkrylov}), but for all methods where convergence
is determined by the degree of the minimal polynomial of $TA$. % of the preconditioned matrix. 

In practical situations, 
the computation of an \textit{ideal} SPD preconditioner $T = \left| A \right|^{-1}$ is prohibitively costly.
However, we show that it is possible to construct inexpensive SPD preconditioners that resemble 
$\left| A \right|^{-1}$ and can
significantly accelerate the convergence of an iterative method.
%We refer to such a preconditioning strategy as the \textit{absolute value preconditioning}
%and define \textit{absolute value preconditioners} as follows.

\begin{definition}\label{def:avp}
We call 
an SPD preconditioner $T$ for a symmetric indefinite linear system $Ax = b$
an \emph{AV preconditioner} if it satisfies
\begin{equation}\label{eqn:avp_spectral}
\delta_0 (v, T^{-1} v) \leq ( v, \left| A \right| v) \leq \delta_1 (v, T^{-1} v), \ \forall v % \in \mathbb{R}^n
\end{equation} 
with constants $\delta_1 \geq \delta_0 > 0$, such that the ratio $\delta_1 / \delta_0 \geq 1$ is reasonably small. 
%
%If the linear system represents a hierarchy of mesh problems,
%then 
%%%%
%we call $T$ \emph{optimal} if the ratio is
%%%%
%%%the ratio must be 
%independent of the problem size $n$, i.e.,
%matrices $\left| A \right|$ and $T^{-1}$ are spectrally equivalent~\cite{Dyakonov:96}. 
\end{definition}

Let us remark that Definition~\ref{def:avp} of the AV preconditioner
is informal because no
precise assumption is made of
how small the ratio $\delta_1 / \delta_0$ should be.
%
%However, qualitively, i
It is clear from~(\ref{eqn:avp_spectral}) that
$\delta_1 / \delta_0$ measures how well the preconditioner $T$ approximates
$\left| A \right|^{-1}$, up to a positive scaling.
If $A$ represents a hierarchy of mesh problems then it is desirable that 
$\delta_1 / \delta_0$ is independent of the problem size. 
%, represents a hierarchy of mesh problems, and $\delta_1 / \delta_2$
%
%
In this case, if $A$ is SPD, 
%represents a hierarchy of mesh problems, and $\delta_1 / \delta_2$
%is independent of the mesh size,
%then $\left| A \right| = A$, 
%and, for mesh problems, Definition~\ref{def:avp} 
Definition~\ref{def:avp} of the AV preconditioner
%%
%of the \emph{optimal} AV preconditioner 
%%
%is consistent with the well known concept of  
is consistent with the well known concept of  
spectrally equivalent preconditioning for SPD
systems; see~\cite{Dyakonov:96}.

The following theorem provides bounds for eigenvalues of the preconditioned matrix
$TA$ in terms of the spectrum of $T \left| A \right| $. 
We note that $T$ and $A$, and thus $TA$ and $T \left| A \right| $,  do not 
in general commute. Therefore, our spectral analysis cannot be based on a traditional 
matrix analysis tool, a basis of eigenvectors. 

\begin{theorem}\label{thm:avp_spect_bounds}
Given a nonsingular symmetric indefinite $A \in \mathbb{R}^{n \times n}$ and
an SPD $T \in \mathbb{R}^{n \times n}$, let 
%\begin{equation}\label{eqn:mu}
$
\mu_1 \leq \mu_2 \leq \ldots \leq \mu_n
$
%\end{equation}
be the eigenvalues of $T \left| A \right|$. 
Then eigenvalues 
%\begin{equation}\label{eqn:lambda}
$
\lambda_1 \leq \ldots \leq \lambda_p < 0 < \lambda_{p+1} \leq \ldots \leq \lambda_n
$
%\end{equation} 
of $TA$ are located in intervals 
\begin{equation}\label{eqn:spectr_bounds}
\begin{array}{cccccl}
 - \mu_{n-j+1}  & \leq & \lambda_j  & \leq & - \mu_{p-j+1},& \ j = 1,\ldots,p; \\
 \mu_{j-p}      & \leq & \lambda_j  & \leq & \mu_j,          &\ j = p+1, \ldots, n.
\end{array}
\end{equation}
\end{theorem}
\begin{proof}
We start by observing that the absolute value of the Rayleigh
quotient of the generalized eigenvalue problem $A v = \lambda \left| A \right| v$ is
bounded by $1$, i.e.,
\begin{equation}\label{eqn:f_of_val}
\left|( v, A v)\right| \leq (v,\left| A \right| v), \ \forall v \in \mathbb{R}^n. 
\end{equation}

Now, we recall that the spectra of matrices $T \left| A \right|$ and $T A$ are given by
the generalized eigenvalue problems $\left| A \right| v = \mu T^{-1} v$ and
$A v = \lambda T^{-1} v$, respectively, and introduce the corresponding Rayleigh quotients
\begin{equation}\label{eqn:rq}
\psi(v) \equiv \frac{(v,\left|A\right| v)}{(v, T^{-1} v)}, \ \phi(v) \equiv \frac{(v,A v)}{(v, T^{-1} v)}, \ v \in \mathbb{R}^n.
\end{equation}

Let us fix any index $j \in \left\{ 1, 2, \ldots, n \right\}$, and denote by $S$ 
an arbitrary subspace of $\mathbb{R}^n$ such that $\mbox{dim}(S) = j$. 
Since inequality~(\ref{eqn:f_of_val}) also holds on $S$, using~(\ref{eqn:rq}) we write  
\begin{equation}\label{eqn:rq_rel}
-\psi(v) \leq \phi(v) \leq \psi(v), \ v \in S.
\end{equation}
Moreover, taking the maxima in vectors $v \in S$, and after that the minima in 
subspaces $S \in S^j =  \left\{ S \subseteq \mathbb{R}^n : \mbox{dim}(S) = j \right\}$,
%all $j$-dimensional subspaces
%$S \subseteq \mathbb{R}^n$, 
of all parts of~(\ref{eqn:rq_rel}) preserves the inequalities, so
\begin{equation}\label{eqn:maxmin_rq_rel}
\min_{S \in S^j} \max_{v \in S} (-\psi(v)) \leq 
\min_{S \in S^j} \max_{v \in S} \phi(v) \leq 
\min_{S \in S^j} \max_{v \in S} \psi(v).
\end{equation}
%where $S^j = \left\{ S \subseteq \mathbb{R}^n : \mbox{dim}(S) = j \right\}$.
By the Courant-Fischer theorem (see, e.g.,~\cite{Horn.Johnson:90, Parlett:98}) 
for the Rayleigh quotients $\pm \psi(v)$ and $\phi(v)$ defined in~(\ref{eqn:rq}),
we conclude from~(\ref{eqn:maxmin_rq_rel}) that
%\begin{equation}\label{eqn:cf1}
\[
-\mu_{n-j+1} \leq \lambda_j \leq \mu_j. 
\]
%\end{equation}
Recalling that $j$ has been arbitrarily chosen, we obtain the following bounds on the eigenvalues
of $TA$:
\begin{equation}\label{eqn:spectr_bounds1}
\begin{array}{cccccl}
 - \mu_{n-j+1}  & \leq & \lambda_j  & < & 0,& \ j = 1,\ldots,p; \\
         0      & < & \lambda_j  & \leq & \mu_j,          &\ j = p+1, \ldots, n.
\end{array}
\end{equation}

%\min_{S \subseteq \mathbb{R}^n, \mbox{{\small dim}}(S) = j}

Next, in order to derive nontrivial upper and lower bounds for the $p$ negative and $n-p$
positive eigenvalues $\lambda_j$ in~(\ref{eqn:spectr_bounds1}), we use the fact that
eigenvalues $\xi_j$ and $\zeta_j$ of the generalized eigenvalue problems $\left| A \right|^{-1}v = \xi T v$ and 
$A^{-1}v = \zeta T v$ are the reciprocals of the eigenvalues of the problems $\left| A \right| v = \mu T^{-1} v$
and $A v = \lambda T^{-1} v$, respectively, i.e., 
\begin{equation}\label{eqn:xi}
0 < \xi_1 = \frac{1}{\mu_n} \leq \xi_2 = \frac{1}{\mu_{n-1}} \leq \ldots \leq \xi_n = \frac{1}{\mu_1}, 
\end{equation}
and
\begin{equation}\label{eqn:zeta}
\zeta_1 = \frac{1}{\lambda_p} \leq \ldots \leq  \zeta_p = \frac{1}{\lambda_1} < 0 < \zeta_{p+1} = \frac{1}{\lambda_n} \leq \ldots \leq \zeta_n = \frac{1}{\lambda_{p+1}}. 
\end{equation}

Similar to~(\ref{eqn:f_of_val}),
%\begin{equation}\label{eqn:f_of_val_inv}
\[
\left|( v, A^{-1} v)\right| \leq (v,\left| A \right|^{-1} v), \ \forall v \in \mathbb{R}^n.
\]
%\end{equation} 
Thus, we can use the same arguments as those following~(\ref{eqn:f_of_val})
to show that relations~(\ref{eqn:rq_rel}) and~(\ref{eqn:maxmin_rq_rel}),
with a fixed $j \in \left\{ 1, 2, \ldots, n \right\}$, also hold for
\begin{equation}\label{eqn:rq_inv}
\psi(v) \equiv \frac{(v,\left|A\right|^{-1} v)}{(v, T v)}, \ \phi(v) \equiv \frac{(v,A^{-1} v)}{(v, T v)}, \ v \in \mathbb{R}^n,
\end{equation}
where $\psi(v)$ and $\phi(v)$ are now the Rayleigh quotients of
the generalized eigenvalue problems $\left| A \right|^{-1}v = \xi T v$ and  $ A^{-1}v = \zeta T v$, respectively.
The Courant-Fischer theorem  
for $\pm \psi(v)$ and $\phi(v)$ in~(\ref{eqn:rq_inv})
allows us to conclude from~(\ref{eqn:maxmin_rq_rel}) that
\[
-\xi_{n-j+1} \leq \zeta_j \leq \xi_j. 
\]
Given the arbitrary choice of $j$ in the above inequality, 
by (\ref{eqn:xi})--(\ref{eqn:zeta})
we get the following bounds on the eigenvalues
of $TA$:
\begin{equation}\label{eqn:spectr_bounds2}
\begin{array}{cccccl}
 - 1/\mu_{p-j+1}  & \leq & 1/\lambda_j  & < & 0,& \ j = 1,\ldots,p; \\
         0      & < &  1/\lambda_j  & \leq & 1/\mu_{j-p},          &\ j = p+1, \ldots, n.
\end{array}
\end{equation}
Combining~(\ref{eqn:spectr_bounds1}) and~(\ref{eqn:spectr_bounds2}), we obtain~(\ref{eqn:spectr_bounds}).
%This completes the proof.  
\qquad\end{proof}

Theorem~\ref{thm:avp_spect_bounds} suggests two useful implications
%,which are 
given by the corresponding corollaries below.
In particular, the following result 
describes $\Lambda(TA)$, i.e., the spectrum of the preconditioned matrix $TA$, in terms of 
$\delta_0$ and $\delta_1$ in~(\ref{eqn:avp_spectral}). 

\begin{corollary}\label{cor:avp_spect} 
Given a nonsingular symmetric indefinite $A \in \mathbb{R}^{n \times n}$,
an SPD $T \in \mathbb{R}^{n \times n}$, and constants $\delta_1 \geq \delta_0 > 0$
satisfying~(\ref{eqn:avp_spectral}), we have 
%all the eigenvalues of $TA$ are located in the union of two intervals 
\begin{equation}\label{eqn:cond_prec_matr}
\Lambda(TA) \subset \left[ - \delta_1, -\delta_0 \right] \bigcup \left[ \delta_0, \delta_1 \right],
\end{equation}
where $\Lambda(TA)$ is the spectrum of $TA$.
\end{corollary}
\begin{proof}
Follows directly from~(\ref{eqn:avp_spectral}) and~(\ref{eqn:spectr_bounds}) with $j = 1,p,p+1,n$.
\qquad\end{proof}

The next corollary
shows that
the presence of
reasonably populated clusters of eigenvalues in the spectrum of $T \left| A \right|$
guarantees the occurrence of corresponding 
clusters in the spectrum of the preconditioned matrix $TA$. 

\begin{corollary}\label{cor:avp_cluster} 
Given a nonsingular symmetric indefinite $A \in \mathbb{R}^{n \times n}$ and
an SPD $T \in \mathbb{R}^{n \times n}$,  
let $\mu_l \leq \mu_{l+1} \leq \ldots \leq \mu_{l+k-1}$
be a sequence of $k$ eigenvalues of $T\left|A\right|$, 
where $1 \leq l < l+k-1 \leq n$ and $\tau = \left| \mu_l - \mu_{l+k-1} \right|$.
Then, if $k \geq p+2$, the $k - p$ positive eigenvalues
$\lambda_{l+p} \leq \lambda_{l+p+1} \leq \ldots \leq \lambda_{l+k-1}$ of $TA$
are such that $\left| \lambda_{l+p} - \lambda_{l+k-1}\right| \leq \tau$.
Also, if $k \geq (n-p)+2$, the $k - (n-p)$ negative eigenvalues 
$\lambda_{n-k-l+2} \leq \ldots \leq \lambda_{p-l} \leq \lambda_{p-l+1}$
of $TA$ are such that $\left| \lambda_{n-k-l+2} - \lambda_{p-l+1} \right| \leq \tau$.
\end{corollary}
\begin{proof}
Follows directly from bounds~(\ref{eqn:spectr_bounds}).
\qquad\end{proof}

Corollary~\ref{cor:avp_spect} implies that the ratio $\delta_1 / \delta_0 \geq 1$ of the constants 
from~(\ref{eqn:avp_spectral}) measures the quality of the AV
preconditioner $T$. Indeed, the convergence speed of the
preconditioned minimal residual method is determined by the spectrum of $TA$,
primarily by the intervals of the right-hand side of inclusion \eqref{eqn:cond_prec_matr}.
Additionally, Corollary~\ref{cor:avp_cluster} prompts that a ``good'' AV 
preconditioner should ensure clusters of eigenvalues in the spectrum of $T\left|A\right|$.
This implies the clustering of eigenvalues of the preconditioned matrix $TA$, which has a favorable
effect on the convergence behavior of a polynomial iterative method, such as PMINRES.
%the preconditioned minimal
%residual method. %, for linear system~(\ref{eqn:sys}).     
%

%The above discussion concerns the construction of absolute value preconditioners for a
%general symmetric indefinite linear system~(\ref{eqn:sys}).
In the next section, we construct an example of the AV 
preconditioner for a particular model problem. We apply the MG techniques.  
%The preconditioner is based 
%on the MG
%special case where $A$
%has a (block) diagonal approximation $C$ such that $\left| C \right|^{-1}$
%is available at a small computational cost. 
%A reasonable absolute value preconditioner can then be expected to be given by $T = \left| C \right|^{-1}$.

% \section{MG absolute value preconditioning for a model problem}\label{subsec:prec_constr}
\section{MG AV preconditioning for a model problem}\label{subsec:prec_constr}
Let us consider the following real boundary value problem, 
%\begin{equation}\label{eqn:helmholtz_bvp}
%\begin{array}{ll}
%-\Delta u (\mathrm x, \mathrm y) - c^2 u(\mathrm x,\mathrm y)  =  f(\mathrm x,\mathrm y), & (\mathrm x,\mathrm y) \in \Omega  = (0,1)\times(0,1),\\
%\\
%u|_\Gamma = 0, &   
%\end{array}
%\end{equation} 
\begin{equation}\label{eqn:helmholtz_bvp}
-\Delta u (\mathrm x, \mathrm y) - c^2 u(\mathrm x,\mathrm y)  =  f(\mathrm x,\mathrm y), \  (\mathrm x,\mathrm y) \in \Omega  = (0,1)\times(0,1), \ u|_\Gamma = 0,    
\end{equation} 
where $\displaystyle \Delta = \partial^2/\partial \mathrm x ^2 + \partial^2 /\partial \mathrm y^2$ is the Laplace operator 
%(Laplacian)
%$c^2 \geq 0$, $f(\mathrm x, \mathrm y) \in \mathcal{C}(\Omega)$, 
and $\Gamma$ denotes the boundary of $\Omega$.
Problem (\ref{eqn:helmholtz_bvp}) is a particular instance of
the Helmholtz equation with Dirichlet boundary conditions, where $c > 0$ is a wave number.  
%see, e.g.,~\cite{Tikhonov.Samarskii:77}.

After introducing a uniform grid of size $h$ in both directions and
using the standard $5$-point finite-difference stencil to discretize continuous 
problem (\ref{eqn:helmholtz_bvp}), 
%see, e.g.,~\cite{Godunov.Ryabenkii:87}, 
one obtains the corresponding discrete problem % $Ax = b$ of the form 
\begin{equation}\label{eqn:helmholtz_fd}
 (L - c^2 I) x = b,
\end{equation} 
where $A \equiv L - c^2 I$ represents a discrete negative Laplacian $L$ (later called ``Laplacian''),
satisfying the Dirichlet boundary condition 
%at the grid points on the boundary 
shifted by a scalar $c^2$.  
%times the identity matrix $I$.
%$n = t^2$ .

The common rule of thumb, see, e.g.,\ \cite{Elman.Ernst.OLeary:01, Harari.Hughes:91}, for discretizing~(\ref{eqn:helmholtz_bvp}) is
\begin{equation}\label{eqn:thumb}
ch \leq \pi/5.
\end{equation}
%
%The right-hand side $b$ in (\ref{eqn:helmholtz_fd}) is the vector of
%function values $f(\mathrm x,\mathrm y)$ calculated at the grid points. 
%%In our numerical tests, $b$ is generated randomly. 
%The solution 
%of system~(\ref{eqn:helmholtz_fd}) provides an approximation
%to the solution of the boundary value problem~(\ref{eqn:helmholtz_bvp}). %, evaluated at the grid points.
Below, we call~(\ref{eqn:helmholtz_fd}) the \textit{model problem}.
%
%Let $\mu_{jk} = \frac{4}{h^2} \left( \sin^2 \frac{j\pi}{2 (t+1)} +  \sin^2 \frac{k\pi}{2 (t+1)} \right)$
%be the eigenvalues of the Laplacian $L$ associated with the eigenvectors
%$v_{jk} = \sin (j\pi x_r) \sin (k \pi y_s)$, where $(x_r, y_s)$ are the grid points
%and $j,k = 1,\ldots,t$. 
%%Let $(\mu_{jk}, v_{jk})$ denote the eigenpairs of the Laplacian $L$, where
%%\begin{equation}\label{eqn:laplace_eval}
%%\mu_{jk} = \frac{4}{h^2} \left( \sin^2 \frac{j\pi}{2 (N+1)} +  \sin^2 \frac{k\pi}{2 (N+1)} \right), \ j,k = 1,\ldots,N,
%%\end{equation}
%%are the eigenvalues and 
%%\begin{equation}\label{eqn:laplace_evec}
%%v_{jk} = \sin (j\pi x_r) \sin (k \pi y_s), \ j,k,r,s = 1,\ldots,N, 
%%\end{equation}
%%are the corresponding eigenvectors; $(x_r, y_s)$ are the grid points. 
%It is clear that $v_{jk}$ are also the eigenvectors of 
%$L - c^2 I$ corresponding to the eigenvalues $\lambda_{jk} = \mu_{jk} - c^2$, i.e., 
%$(\lambda_{jk},v_{jk})$ are the eigenpairs of the shifted Laplacian.
%% 
%%Assuming 
We assume that the shift $c^2$ is different from any eigenvalue of the Laplacian 
and is greater than the smallest but less than the largest eigenvalue.
%, i.e., 
%$\lambda_{\min}(L) <c^2 < \lambda_{\max}(L)$, where $\lambda_{\min}(L) = 2\pi^2 + \mathcal{O}(h^2)$
%and $\lambda_{\max}(L) = 8h^{-2} + \mathcal{O}(1)$. 
%
Thus, the matrix $L - c^2 I$ is nonsingular symmetric indefinite.
%, which allows applying the idea of the AV preconditioning.  
In the following subsection, we apply the idea of the AV preconditioning 
to construct an MG AV preconditioner for system (\ref{eqn:helmholtz_fd}). 
%based on the idea of the absolute value 
%preconditioning introduced in the previous section.
%
%one can choose an absolute value preconditioner; 
%see Definition~\ref{def:avp} with $A = L - c^2 I$.
%i.e.,$T \approx \left|A\right|^{-1} = \left|L - c^2 I\right|^{-1}$.

While our main focus throughout the paper is on the 2D problem~(\ref{eqn:helmholtz_bvp}),
in order to simplify presentation of theoretical analysis, we also refer to the 1D
analogue 
\begin{equation}\label{eqn:helmholtz_bvp1D}
%\begin{array}{ll}
-\displaystyle  u'' (\mathrm x) - c^2 u(\mathrm x)  =  f(\mathrm x), \  u(0) = u(1) = 0.   
%\end{array}
\end{equation}
The conclusions drawn from~(\ref{eqn:helmholtz_bvp1D}), however, 
remain qualitatively the same for the 2D problem of interest, which 
we test numerically.

\subsection{Two-grid AV preconditioner}\label{subsec:2grid}

Along with the fine grid of mesh size $h$ underlying problem~(\ref{eqn:helmholtz_fd}),
let us consider a coarse grid of mesh size $H > h$. We denote the discretization of the
Laplacian on this grid by $L_H$, 
% (further called ``the coarse-level Laplacian''), 
and $I_H$ represents the identity operator of the corresponding dimension.   
We assume that the exact fine-level absolute value $\left|L - c^2 I\right|$ and its inverse  
are not computable, whereas the inverse of the coarse-level operator $\left|L_H - c^2 I_H\right|$ can be efficiently constructed. 
In the two-grid framework, we use the subscript 
$H$ to refer to the quantities defined on the coarse grid. No subscript is used for denoting the fine grid quantities.    

While $\left|L - c^2 I\right|$ is not available, let us assume 
that we have its SPD approximation $B$, i.e., $B \approx |L - c^2 I|$ and $B = B^* > 0$.
The operator $B$ can be given in the explicit matrix form or through the action 
on a vector.
We suggest the following general scheme 
as a two-grid AV preconditioner for model problem~(\ref{eqn:helmholtz_fd}).

\vspace{.1in}
\begin{algorithm}[The two-grid AV preconditioner]\label{alg:g2g}
Input: $r$, $B \approx |L - c^2 I|$. Output: $w$.
%
%Output: w
\begin{enumerate}
	\item \emph{Presmoothing}. Apply $\nu$ smoothing steps, $\nu \geq 1$: % with the zero initial guess ($w^{(0)} = 0$):
\begin{equation}\label{eqn:pre}
w^{(i+1)} = w^{(i)} + M^{-1} (r - B w^{(i)}),  \ i = 0,\ldots,\nu - 1, \ w^{(0)} = 0,
\end{equation}
where $M$ defines a smoother.  
%This results in the vector 
Set $w^{pre} = w^{(\nu)}$. %, $\nu \geq 1$. 
	\item \emph{Coarse grid correction}. Restrict ($R$) $r - B w^{pre}$ to the coarse grid, 
	apply $\left|L_H - c^2 I_H\right|^{-1}$, 
	and prolongate ($P$) to the fine grid.
	This delivers the coarse grid correction, which is added to $w^{pre}$: % to obtain the corrected vector $w^{cgc}$:
%	\begin{equation}\label{eqn:cgc}
\begin{eqnarray}
\label{eqn:cgc-1}  w_H & = &  \left|L_H - c^2 I_H \right|^{-1} R \left(r - B w^{pre}\right), \\
\label{eqn:cgc-2}	 w^{cgc} & = & w^{pre} + P w_H.
\end{eqnarray}
%	\end{equation}
%where $P$ and $R$ are prolongation and restriction operators, respectively. 
	\item \emph{Postsmoothing}. Apply $\nu$ smoothing steps: % with the initial guess $w^{(0)} = w^{cgc}$:
\begin{equation}\label{eqn:post}
w^{(i+1)} = w^{(i)} + M^{-*} (r - B w^{(i)}),  \ i = 0,\ldots,\nu - 1, \ w^{(0)} = w^{cgc},
\end{equation}
where $M$ and $\nu$ are the same as in step~1. 
Return $w = w^{post} = w^{(\nu)}$. 
%Set $w^{post} = w^{(\nu)}$.  
 \end{enumerate}
\end{algorithm}
\vspace{.1in}

In (\ref{eqn:cgc-1}) we assume that $\left|L_H - c^2 I_H \right|$ is nonsingular,
i.e., $c^2$ is different from any eigenvalue of $L_H$. 
%The number of smoothing steps in~(\ref{eqn:pre}) and~(\ref{eqn:post})
%is the same; 
The presmoother is defined by the nonsingular $M$, while the postsmoother is
delivered by $M^{*}$.
Note that the (inverted) absolute value appears only on the coarse grid, while
the fine grid computations are based on the approximation~$B$.

It is immediately seen that if $B = |L - c^2 I|$, Algorithm~\ref{alg:g2g} represents 
a formal two-grid cycle~\cite{Briggs.Henson.McCormick:00, Trottenberg.Oosterlee.Schuller:01} 
for system
\begin{equation}\label{eqn:av_sys}
\left|L - c^2 I\right| z = r.
\end{equation} 
Note that the introduced scheme is rather general in that different choices
of approximations $B$ and smoothers $M$ lead to different preconditioners. We address these
choices in more detail in the following subsections.  
%further in the paper.  

%\begin{equation}\label{eqn:av_lowrank}
%|A| = A +  2V_p |\Lambda_p| V_p^*, 
%\end{equation}

%Algorithm~\ref{alg:g2g} can be viewed as a modified two-grid cycle for linear system $\left|L - c^2 I\right| z = r$. 
%%with $A = L - c^2 I$, 
%Here, the computationally expensive or unavailable absolute value of the fine-grid operator, $\left| L - c^2 I \right|$, 
%is replaced by the easily accessible negative Laplacian 
%$L = \left|L\right|$ at
%smoothing steps~(\ref{eqn:pre}) and~(\ref{eqn:post}), 
%as well as in the restricted residual in~(\ref{eqn:cgc-1}).
%The operator $L - c^2 I$ appears only at the coarse grid~(\ref{eqn:cgc-1}).
%
%If used repeatedly, 
%Algorithm~\ref{alg:g2g} does not solve any linear system. 
%%with coefficient matrices
%%~$\left|L - c^2 I\right| z = r$.
%However, 
%as seen later, its use as a preconditioner (along with the respective MG extension described below) 
%significantly accelerates the convergence of an iterative scheme
%applied to model problem~(\ref{eqn:helmholtz_fd}) with a relatively small shift value; moreover, the
%convergence is independent of the mesh size. 

%Two-grid 
It can be verified that the AV preconditioner given by Algorithm~\ref{alg:g2g} implicitly
constructs a mapping $r \mapsto w = T_{tg} r$, where the operator is %$T = T_{tg}$ has the following structure:
\begin{equation}\label{eqn:2grid_struct}
T_{tg}  =  \left(I - M^{-*} B \right)^{\nu} P \left| L_H - c^2 I_H \right|^{-1} R \left(I -  B M^{-1} \right)^{\nu} + F,
\end{equation}
with $F = B^{-1} - \left(I - M^{-*} B \right)^{\nu} B^{-1}\left(I -  B M^{-1}\right)^{\nu}$. 
The fact that the constructed preconditioner $T = T_{tg}$ 
is SPD follows directly from the observation that the first term in~(\ref{eqn:2grid_struct}) is
SPD provided that $P = \alpha R^*$ for some nonzero scalar $\alpha$, while
the second term $F$ is SPD if the spectral radii of $I - M^{-1} B$ and
$I - M^{-*} B$ are less than $1$. The latter condition requires the pre- and postsmoothing iterations~(\ref{eqn:pre}) 
and~(\ref{eqn:post}) to represent convergent methods 
for 
%system 
%system (\ref{eqn:helmholtz_fd}) with $c = 0$ and $b = r$ 
%(i.e., for 
% the discrete Poisson equation 
%\begin{equation}\label{eqn:poisson}
$
B y = r.
$
%\end{equation}
%on their own. 
%We note that 
Note that 
the above argument 
%for the operator $T = T_{tg}$ to be SPD 
essentially repeats the one used to justify symmetry and positive definiteness of a 
preconditioner based on the
standard 
two-grid cycle for an SPD system; see, e.g.,~\cite{Bramble.Zhang:00, Tatebe:93}.  
%applied within an iterative scheme, e.g., the preconditioned conjugate gradient method (PCG), to solve an SPD system; 
%see, e.g., \cite{Bramble.Zhang:00, Tatebe:93}.  

In this paper we consider two different choices of the approximation $B$.
The first choice is given by $B = L$, i.e., it is suggested to approximate 
the absolute value $|L-c^2 I|$ by the Laplacian $L$. The second choice
is delivered by $B = p_{m} (L - c^2 I)$, where $p_{m}$ is a polynomial
of degree at most $m$ such that $p_m(L - c^2 I) \approx |L - c^2 I|$.              
\subsection{Algorithm~\ref{alg:g2g} with $B =  L$}\label{subsec:BLaplacian}
If $B = L$, Algorithm~\ref{alg:g2g} 
%can be viewed from two different perspectives.
%On the one hand, it 
can be regarded as a step of a standard two-grid
method~\cite{Briggs.Henson.McCormick:00, Trottenberg.Oosterlee.Schuller:01} applied
to the Poisson equation
\begin{equation}\label{eqn:poisson}
L y = r,
\end{equation}
modified by replacing the operator $L_H$ by $|L_H - c^2 I_H|$
on the coarse grid. The question remains if the algorithm 
%Algorithm~\ref{alg:g2g}
delivers a form of an approximate solve for absolute value problem~(\ref{eqn:av_sys}),
and hence is suitable for AV preconditioning of~(\ref{eqn:helmholtz_fd}).  
To be able to answer this question, we analyze 
the propagation of the initial error 
$e_0^{\mbox{{\tiny AV}}} = |L - c^2 I|^{-1} r$
of~(\ref{eqn:av_sys}) under the action of the algorithm. %Algorithm~\ref{alg:g2g}.
%i.e.,
%we consider the two-grid preconditioner as an inexact solve for 
%the absolute value problem~(\ref{eqn:av_sys}).   
%
%We start by relating errors of equations~(\ref{eqn:av_sys}) and~(\ref{eqn:poisson}).

We start by relating errors of~(\ref{eqn:av_sys}) and~(\ref{eqn:poisson}).
\begin{lemma}\label{lem:error}
Given a vector $w$, consider errors $e^{\mbox{{\tiny AV}}}(w) = |L-c^2I|^{-1} r - w$ 
and $e^{\mbox{{\tiny P}}} (w) = L^{-1}r - w$
for~(\ref{eqn:av_sys}) and~(\ref{eqn:poisson}), respectively. 
Then 
\begin{equation}\label{eqn:error}
e^{\mbox{{\tiny AV}}} (w)=  e^{\mbox{{\tiny P}}} (w)+  (c^2 I - W_p) L^{-1} |L - c^2I|^{-1}r,
\end{equation}
where $W_p = 2 V_p |\Lambda_p| V_p^*$, $V_p$ is the matrix of eigenvectors of $L - c^2 I$
corresponding to the $p$ negative eigenvalues $\lambda_1 \leq \ldots \leq \lambda_p < 0$, and 
$|\Lambda_p| = \mbox{diag}\left\{ |\lambda_1|, \ldots, |\lambda_p| \right\}$.
\end{lemma}
\begin{proof}
Observe that for any $w$,
\begin{eqnarray}
\nonumber e^{\mbox{{\tiny AV}}} (w)& = & |L - c^2 I|^{-1} r  - w =  |L - c^2 I|^{-1} r  + (e^{\mbox{{\tiny P}}}(w) - L^{-1}r) \\
\nonumber                       & = & e^{\mbox{{\tiny P}}}(w) + |L - c^2 I|^{-1} L^{-1} (L - |L - c^2 I|)r. 
\end{eqnarray}
Denoting  $A = L - c^2 I$, we use the expression $|A| = A -  2V_p \Lambda_p V_p^*$ to get~(\ref{eqn:error})  
\end{proof}

Algorithm~\ref{alg:g2g}
transforms the initial error
$e_0^{\mbox{{\tiny P}}} = L^{-1}r$ of equation~(\ref{eqn:poisson}) 
into
\begin{equation}\label{eqn:eP}
e^{\mbox{{\tiny P}}} = S_2^{\nu} K S_1^{\nu} e_0^{\mbox{{\tiny P}}},
\end{equation}
where $S_1 = I - M^{-1} L$ and $S_2 =  I - M^{-*} L$ are pre- and postsmoothing
operators, $K = I - P|L_H - c^2 I_H|^{-1} R L$ corresponds to the coarse grid
correction step, and $e^{\mbox{{\tiny P}}} = L^{-1} r - w^{post}$.  
%
%On the other hand, %the algorithm 
%Algorithm~\ref{alg:g2g} %can be viewed as 
%represents a two-grid cycle for linear system 
%\begin{equation}\label{eqn:av_sys}
%\left|L - c^2 I\right| z = r,
%\end{equation} 
%%with $A = L - c^2 I$, 
%where the computationally expensive or unavailable absolute value of the fine-grid operator, $\left| L - c^2 I \right|$, 
%is substituted (approximated) by the easily accessible
%Laplacian $L = \left|L\right|$ at
%smoothing steps~(\ref{eqn:pre}),~(\ref{eqn:post}) and 
%%as well as 
%in the residual in~(\ref{eqn:cgc-1}).
%%Below, we consider the multigrid components of 
%%Algorithm~\ref{alg:g2g}
%%, i.e.,
%%the smoothing and coarse grid correction steps, 
%%in more detail.
%% effects of this substitution in more detail.
%%The following discussion shows that, nevertheless, this
%%allows constructing mesh-independent 
%%SPD preconditioners for model problem~(\ref{eqn:helmholtz_fd})
%%with relatively small shift values. 
%
%on the initial error $e_0^{\mbox{{\tiny AV}}} = |L - c^2 I|^{-1} r$
%for the absolute value system~(\ref{eqn:av_sys}). 
%
%
%The following lemma relates $e^{\mbox{{\tiny AV}}}$ to the error
%$e^{\mbox{{\tiny P}}}$ in 
% = L^{-1}r  - w^{(i)}$ for Poisson equation~(\ref{eqn:poisson})
%evaluated at the same vector $w^{(i)}$. 
%The superscripts ``AV'' and ``P'' abbreviate ``absolute value'' and ``Poisson'', respectively. 
%Clearly, the analysis below also holds for the case of postsmoothing~(\ref{eqn:post}) 
%with $w^{pre}$ replaced by $w^{post}$.
%
Denoting the error of absolute value system~(\ref{eqn:av_sys})
after applying Algorithm~\ref{alg:g2g} by
$e^{\mbox{{\tiny AV}}} = |L - c^2 I|^{-1} r - w^{post}$ and
observing that $e_0^{\mbox{{\tiny P}}} = |L-c^2 I| L^{-1}e_0^{\mbox{{\tiny AV}}}$,
by~(\ref{eqn:error})--(\ref{eqn:eP}) we obtain 
% 
%combining~(\ref{eqn:eP}) and~(\ref{eqn:error}), 
%and observing that $e_0^{\mbox{{\tiny P}}} = |L-c^2 I| L^{-1}e_0^{\mbox{{\tiny AV}}}$, 
%we obtain the
%following expression for the error: 
\begin{equation}\label{eqn:eAV}
e^{\mbox{{\tiny AV}}} = \left( S_2^{\nu} K S_1^{\nu}|L-c^2 I| 
+ c^2 I - W_p \right) L^{-1}e_0^{\mbox{{\tiny AV}}}.
\end{equation}
The last expression gives an explicit form of the desired error propagation operator,
which we denote by $G$: 
\begin{equation}\label{eqn:G}
G = \left( S_2^{\nu} K S_1^{\nu}|L-c^2 I| + c^2 I - W_p \right) L^{-1}.
\end{equation}

Below, as a smoother, 
%in Algorithm~\ref{alg:g2g}, 
%we use the $\omega$-damped Jacobi iteration, i.e.,
%we use 
we use a simple Richardson's iteration,
i.e., $S_1 = S_2 = I - \tau L$, 
%in~(\ref{eqn:eAV}) and~(\ref{eqn:G}),
where $\tau$ is an iteration parameter. 
%, $\tau = \omega h^2/2$. 
The restriction $R$ is given by the full weighting
and the prolongation $P$ by the standard piecewise linear interpolation; see~\cite{Briggs.Henson.McCormick:00, Trottenberg.Oosterlee.Schuller:01}.  
%is used 
%for the restriction $R$  
%and the standard piecewise linear interpolation for the prolongation $P$; 
%see~\cite{Briggs.Henson.McCormick:00, Trottenberg.Oosterlee.Schuller:01}.  

At this point, in order to simplify further presentation, let us refer to the
one-dimensional analogue~(\ref{eqn:helmholtz_bvp1D})
of model problem~(\ref{eqn:helmholtz_bvp}).  
%\begin{equation}\label{eqn:helmholtz_bvp1D}
%%\begin{array}{ll}
%-\displaystyle  u^{''} (\mathrm x) - c^2 u(\mathrm x)  =  f(\mathrm x), \  u(0) = u(1) = 0.   
%%\end{array}
%\end{equation}
%
In this case, the matrix $L$ is
tridiagonal: $L = \mbox{tridiag}\left\{ -1/h^2,2/h^2,-1/h^2 \right\}$.
We assume that $n$, the number of interior grid nodes, is odd: $h = 1/(n+1)$. 
The coarse grid is then obtained by dropping the odd-numbered nodes.
We denote 
the size of the coarse grid problem by $N = (n+1)/2-1$; $H = 1/(N+1) = 2h$.
The tridiagonal matrix $L_H$ denotes the discretization of the 1D
Laplacian on the coarse level.   

Recall that the eigenvalues of $L$ are $\theta_j = \frac{4}{h^2} \sin^2 \frac{j \pi h}{2}$
with corresponding eigenvectors $v_j = \sqrt{2h} \left[ \sin l j \pi h \right]_{l=1}^n$.
Similarly, the eigenvalues of $L_H$ are $\theta_j^H = \frac{4}{H^2} \sin^2 \frac{j \pi H}{2}$, 
and the coarse grid eigenvectors are denoted by 
$v_j^H = \sqrt{2H} \left[ \sin l j \pi H \right]_{l=1}^N$. 
It is clear that operators $L - c^2 I$ and $L_H-c^2 I_H$ 
have the same sets of eigenvectors as $L$ and $L_H$
with eigenvalues $t_{j} = \theta_{j} - c^2$ and 
$t_{j}^H = \theta_{j}^H - c^2$, respectively.

Let $e_0^{\mbox{{\tiny AV}}} = \sum_{j=1}^n \alpha_j v_j$ 
be the expansion of the initial error in the eigenbasis of~$L$.   
Since $e^{\mbox{{\tiny AV}}}  = G e_0^{\mbox{{\tiny AV}}} = \sum_{j=1}^n \alpha_j (G v_j)$,
we are interested in the action of the error propagation operator 
(\ref{eqn:G}) on the eigenmodes $v_j$. 
\iftoggle{report}{We first consider the images of $v_j$ under the action of the
coarse grid correction operator $K$.

The action of the full weighting operator on eigenvectors $v_j$ is well known~\cite{Briggs.Henson.McCormick:00}:   
\begin{equation}\label{eqn:Rv}
R v_j = \left\{ 
\begin{array}{ll}
\displaystyle c_j^{2} v_j^{H}, & j = 1,\ldots,N, \\
\displaystyle 0, & j = N+1, \\
\displaystyle -c_j^2 v^{H}_{n+1 -j}, & j = N+2, \ldots,n.
\end{array}
\right.
\end{equation}  
Here, $c_j = \cos \frac{j\pi h}{2}$. With $s_j = \sin \frac{j \pi h}{2}$,
the action of the piecewise linear interpolation is written as 
\begin{equation}\label{eqn:Pv}
P v_j^H = c_j^2 v_j - s_j^2 v_{n+1-j}, \ j = 1,\ldots,N,
\end{equation}
where $v_{n+1-j}$ is the so-called \textit{complementary} mode of $v_j$~\cite{Briggs.Henson.McCormick:00}.
Combining~(\ref{eqn:Rv}) and~(\ref{eqn:Pv}) leads to the following expression for $Kv_j$:
} % end \iftoggle{report} 
{

The action of the operators $R$ and $P$ on $v_j$ and $v^{H}_j$, respectively, 
is well known; see, e.g.,~\cite[pp.~80--81]{Briggs.Henson.McCormick:00}. Thus, 
it is easy to obtain the following expression for~$Kv_j$:    
}
\begin{equation}\label{eqn:Kv}
K v_j = \left\{ 
\begin{array}{ll}
\displaystyle \left( 1 - c_j^{4} \frac{\theta_j}{|t_j^H|} \right) v_j + s_j^2 c_j^2 \frac{\theta_j}{|t_j^{H}|} v_{n+1-j}, & j = 1,\ldots,N, \\
\displaystyle v_j, & j = N+1, \\
\displaystyle \left( 1 - c_j^4 \frac{\theta_j}{|t^H_{n+1-j}|}\right)v_j   + s_j^2 c_j^2 \frac{\theta_j}{|t_{n+1-j}^H|}v^{H}_{n+1 -j}, & j = N +2, \ldots,n.
\end{array}
\right.
\end{equation}  
\iftoggle{report}{}
{
Here, $c_j = \cos \frac{j\pi h}{2}$ and $s_j = \sin \frac{j \pi h}{2}$.
}Since $v_j$ are the eigenvectors of $S_1 = S_2 = I - \tau L$, 
$L-c^2I$, $L^{-1}$ and $W_p$,~(\ref{eqn:G}) leads to 
explicit expressions for $Gv_j$.
\begin{theorem}\label{thm:G}
%Let $S_1 = S_2 = I - \tau L$. 
Let $c^2 < \theta_{N+1} = 2/h^2$. Then the error propagation operator $G$ in~(\ref{eqn:G}) acts on the 
eigenvectors $v_j$ of 1D Laplacian as follows: 
\begin{equation}\label{eqn:Gv}
G v_j = \left\{ 
\begin{array}{ll}
\displaystyle  g_j^{(11)} v_j + g_j^{(12)} v_{n+1-j}, & j = 1,\ldots,N, \\
\displaystyle  g_j v_j, & j = N+1, \\
\displaystyle g_j^{(21)} v_j   + g_j^{(22)} v_{n+1 -j}, & j = N +2, \ldots,n,
\end{array}
\right.
\end{equation}  
where
\begin{eqnarray}
\label{eqn:g11} 
g_j^{(11)} & = &(1 - \tau \theta_j)^{2\nu} \left( 1 - c_j^4 \frac{\theta_j}{|t_j^H|}\right) \frac{|t_j|}{\theta_j} + \frac{c^2}{\theta_j} - \frac{\beta_j}{\theta_j}  \; , \\ 
\label{eqn:g12}
g_j^{(12)} & = & (1 - \tau \theta_j)^{\nu} s_j^2 c_j^2 \frac{|t_j|}{|t_j^{H}|} (1 - \tau \theta_{n+1-j})^{\nu} \; ,  \\
\label{eqn:g}
g_j & =  & (1-\tau \theta_j)^{2 \nu} \frac{|t_j|}{\theta_j} + \frac{c^2}{\theta_j} \; , \\
\label{eqn:g21}
g_j^{(21)} & = &(1 - \tau \theta_j)^{2\nu} \left( 1 - c_j^4 \frac{\theta_j}{|t_{n+1-j}^H|}\right) \frac{|t_j|}{\theta_j} + \frac{c^2}{\theta_j} \;, \\ 
\label{eqn:g22}
g_j^{(22)} & = & (1 - \tau \theta_j)^{\nu} s_j^2 c_j^2 \frac{|t_j|}{|t_{n+1-j}^{H}|} (1 - \tau \theta_{n+1-j})^{\nu} \; ; 
\end{eqnarray}
%and 
%\begin{equation}\label{eqn:beta}
and 
$\beta_j = \left\{ 
\begin{array}{ll}
%\displaystyle  2|\lambda_j|, & \theta_j < c^2, \\
\displaystyle  2(c^2 - \theta_j), & \theta_j < c^2, \\
\displaystyle  0, & \theta_j > c^2 
\end{array}
\right.
$.
%\end{equation}  
% \left( 1 - c_j^{4} \frac{\mu_j}{|\lambda_j^H|} \right)
\end{theorem}

\vspace{.07in}

Theorem~\ref{thm:G} implies that for relatively small shifts, 
Algorithm~\ref{alg:g2g} with $B = L$ and a proper choice of $\tau$ and $\nu$ 
reduces the error of~(\ref{eqn:av_sys}) in the directions of almost 
all eigenvectors $v_j$. In a few directions, however, the error may be amplified. 
These directions are given by the smooth eigenmodes associated
with $\theta_j$ 
%in the vicinity of the shift 
that are close to~$c^2$ on the right, as well as 
with $\theta_j$ that are distant from $c^2$ on the left. The number of the latter,
if any, is small if $ch$ is sufficiently small, and becomes larger as $ch$ increases.
%We discuss these findings in more detail below.   
%Additionally, as has been
%previously noted, if the number of smoothing steps $\nu$ is not large enough,
%the error may be amplified in the directions of a few oscillatory 
%eigenmodes, complementary to the smooth $v_j$ for which $\theta_j^H \approx c^2$.         

%Let $\omega = 2/3$, so that $\tau = h^2/3$. 
Indeed, let $\tau = h^2/3$, so that $|1 - \tau \theta_j| < 1$ for all $j$ and 
$|1 - \tau \theta_j| < 1/3$ for $j>N$. 
%i.e., 
This choice of the parameter provides 
%It is well known that this choice gives
%%The latter represents 
the least uniform bound for $|1 - \tau \theta_j|$ 
that correspond to the oscillatory eigenmodes~\cite[p.415]{Saad:03}. 
%
%In this case, 
It is then readily seen that~(\ref{eqn:g12}) and~(\ref{eqn:g22})
can be made arbitrarily small within a reasonably small number~$\nu$ of smoothing steps.
Similarly,~(\ref{eqn:g}) and~(\ref{eqn:g21}) can be made arbitrarily close to~$c^2/\theta_j<1$.
%where the inequlaity follows from the theorem assumption. 
%Note that 
If $c^2 << \theta_{N+1}$, then $c^2/\theta_j$ in~(\ref{eqn:g}) and~(\ref{eqn:g21})
is close to zero. 
Thus, 
%the theorem 
Theorem~\ref{thm:G} 
shows that for relatively small shifts, smoothing  
% of Algorithm~\ref{alg:g2g} with $B =  L$
provides small values of~(\ref{eqn:g12})--(\ref{eqn:g22})
and, hence, damps of the oscillatory part of the error. 
Note that the damping occurs even though the smoothing is performed with respect to~(\ref{eqn:poisson}),
not~(\ref{eqn:av_sys}).

Now let us consider~(\ref{eqn:g11}). 
%, which is not significantly reduced by smoothing.
%
Theorem~\ref{thm:G} shows that if 
%reveals a potential difficulty that is given by the case where 
$c^2$ is close to an eigenvalue $\theta_j^H$ of the coarse-level Laplacian, i.e., if
$t_j^H \approx 0$, then
%% 
%the magnitude of the 
the corresponding reduction coefficient~(\ref{eqn:g11}) 
can be large. This means that Algorithm~\ref{alg:g2g} with $B=L$ has a potential
difficulty of amplifying  
%meaning that 
the error in the directions of a few smooth eigenvectors.
%Additionally,
%if $\nu$ is too small, the corresponding coefficients~(\ref{eqn:g12}),~(\ref{eqn:g21}), and~(\ref{eqn:g22}) 
%may be large, i.e., 
%the error may also be amplified in the direction of the complementary oscillatory eigenmode. 
%%
%Note that a 
Similar effect is known to appear for
standard MG methods applied to Helmholtz type problems; %~(\ref{eqn:helmholtz_fd});
see~\cite{Brandt.Taasan:86, Elman.Ernst.OLeary:01}. 
Below, we analyze~(\ref{eqn:g11}) in more detail. % for $\theta^H_j$ (and $\theta_j$) that are distant from $c^2$. 

%Now let us consider~(\ref{eqn:g11}) in more detail.
Let $\theta_j > c^2$. 
Then, using the relation $\theta_j^H = c_j^2 \theta_j$, we can write~(\ref{eqn:g11}) as
%the expression for $g_j^{(11)}$ takes the form
%\begin{eqnarray}
%\nonumber g_j^{(11)} & = &(1 - \tau \mu_j)^{2\nu} \left( 1 - c_j^4 \frac{\mu_j}{|\lambda_j^H|}\right) \frac{|\lambda_j|}{\mu_j} + \frac{c^2}{\mu_j} \\ 
%\nonumber            & = &(1 - \tau \mu_j)^{2\nu} \left( 1 - c_j^2 \frac{1}{|1 - (c^2 / \mu_j^H )|}\right) \left( 1 - \frac{c^2}{\mu_j} \right) + \frac{c^2}{\mu_j},
%\end{eqnarray}
\[
g_j^{(11)}  = (1 - \tau \theta_j)^{2\nu} \left( 1 - c_j^2 \frac{1}{|1 - c^2 / \theta_j^H |}\right) \left( 1 - \frac{c^2}{\theta_j} \right) + \frac{c^2}{\theta_j} \; .
\]
%where $\theta_j^H = c_j^2 \theta_j$ by the definition of $\theta_j$ and $\theta_j^H$.
%
Here, it is easy to see that as $c^2/\theta_j^H \rightarrow 0$, 
%and hence $c^2/\theta_j \rightarrow 0$, 
$g_j^{(11)} \rightarrow (1-\tau \theta_j)^{2\nu}s_j^2 < 1/2$,
meaning that
%i.e.,
%
%In other words, 
%reduction coefficients~(\ref{eqn:g11}) corresponding to 
%$\theta_j$, such that 
%Thus, 
the smooth eigenmodes corresponding to
$\theta_j$ away from $c^2$ 
on the right 
are well~damped. 
%by the coarse grid
%correction. 
%, are significantly 
%smaller than $1$.
% enough
%to ensure efficient error reduction in the direction of the associated smooth eigenvectors.  
%
%Note that for $j \leq N$, $|1 - \tau \theta_j|$ are close to $1$, i.e.,~(\ref{eqn:g11}) is not
%much influenced by smoothing.  

If $\theta_j < c^2$, then~(\ref{eqn:g11}) takes the form 
%can be written as
%\begin{eqnarray}
%\nonumber g_j^{(11)} & = &(1 - \tau \mu_j)^{2\nu} \left( 1 - c_j^4 \frac{\mu_j}{|\lambda_j^H|}\right) \frac{|\lambda_j|}{\mu_j} + \left( 2 - \frac{c^2}{\mu_j} \right)\\ 
%\nonumber            & = &(1 - \tau \mu_j)^{2\nu} \left( 1 - c_j^2 \frac{1}{|1 - (c^2/\mu_j^H)|}\right) \left( \frac{c^2}{\mu_j} - 1\right) + \left( 2 - \frac{c^2}{\mu_j} \right).
%\end{eqnarray}
\[
%g_j^{(11)}  = (1 - \tau \theta_j)^{2\nu} \left( 1 - c_j^2 \frac{1}{(c^2/\theta_j^H) - 1}\right) \left( \frac{c^2}{\theta_j} - 1\right) + \left( 2 - \frac{c^2}{\theta_j} \right).
%\begin{equation}\label{eqn:g11_new}
g_j^{(11)}  = (1 - \tau \theta_j)^{2\nu} \left(\frac{c^2/\theta_j - c_j^2 - c_j^4}{c^2/\theta_j - c_j^2}\right) \left( \frac{c^2}{\theta_j} - 1\right) + \left( 2 - \frac{c^2}{\theta_j} \right).
%\end{equation}
\]
Since $c_j^2 \in (1/2,1)$, for any $c^2/\theta_j > 1$, we can obtain the bound
\[
\frac{c^2/\theta_j - 2}{c^2/\theta_j - 1} \leq \frac{c^2/\theta_j - c_j^2 - c_j^4}{c^2/\theta_j - c_j^2} \leq \frac{c^2/\theta_j - 3/4}{c^2/\theta_j-1/2} \; . 
\]
Additionally, $3^{-2 \nu} < (1 - \tau \theta_j)^{2\nu} < 1$.
%provided that $\tau = h^2/3$. 
%The inequlaities lead to the bound % on~(\ref{eqn:g11_new}): % in terms of $c^2/\theta_j$,
Thus,
\[
l_j < g_j^{(11)} < \frac{3(c^2/\theta_j)-1}{4(c^2/\theta_j)-2} \; , 
%\qquad
%l_j = \left\{ 
%\begin{array}{ll}
%%\displaystyle  2|\lambda_j|, & \theta_j < c^2, \\
%\displaystyle  0, & 1 < c^2/\theta_j < 2 \; , \\ 
%\displaystyle  \frac{8}{9}(2 - c^2/\theta_j), &  c^2/\theta_j \geq  2 \; .
%\end{array}
%\right.
\]
where $l_j = 0$ if $1 < c^2/\theta_j \leq 2$, and $l_j = 2 - c^2/\theta_j$ if $c^2/\theta_j > 2$.

The inequality implies that $|g_j^{(11)}|<1$ for $1 < c^2 / \theta_j \leq 3$, i.e.,
the algorithm 
%coarse grid correction 
reduces the error in the directions of several smooth
eigenvectors associated with $\theta_j$ to the left of $c^2$. 
%
%as $c^2/\theta_j \rightarrow 2$, $g_j^{(11)} \rightarrow \displaystyle (1 - \tau \theta_j)^{2\nu} \left( 1 - c_j^2 / |1 - 2/c_j^2|\right)$,
%where
%%\[ 
%$
%\displaystyle 0 < 1 - c_j^2/|1 - (2/c_j^2)| < 5/6.
%$
%%\]
%
At the same time, we note that as $c^2/\theta_j \rightarrow \infty$, $g_j^{(11)} \rightarrow \infty$, i.e.,
%
%Therefore, 
the smooth eigenmodes corresponding to $\theta_j$ 
that are distant from $c^2$ on the left can be amplified. 
%
%Summarizing the findings of this subsection, we have shown that 
%for sufficiently small $ch$, Algorithm~\ref{alg:g2g} with $B = L$ and a proper choice of 
%the parameters $\tau$ and $\nu$ reduces the error of~(\ref{eqn:av_sys}) in the directions of almost 
%all eigenvectors $v_j$.
%In a few directions, however, the error may possibly be amplified. 
%These directions are given by the smooth eigenmodes associated
%with several $\theta_j$ in the vicinity of the shift $c^2$, as well as 
%$\theta_j$ that are distant from $c^2$ on the left. Additionally, as has been
%previously noted, if the number of smoothing steps $\nu$ is not large enough,
%the error may be amplified in the directions of a few oscillatory 
%eigenmodes, complementary to the smooth $v_j$ for which $\theta_j^H \approx c^2$.         
%
Clearly, if $ch$ is sufficiently small then the number of such
error components is not large (or none), and grows as $ch$ increases. 
%the number of such error components
%may increase. 
%However, 
%if $ch$ is sufficiently small, 
%%the amount of such ``pathological'' 
%%eigencomponents tends to be small, i.e., the total number
%%of the corrupted eigenmodes is small relative to the problem size. 
%%%

%%%This allows concluding that 
%then 
The above analysis shows that Algorithm~\ref{alg:g2g} with $B = L$ 
%%%can be interpreted 
indeed represents a solve for~(\ref{eqn:av_sys}), 
where the solution is approximated everywhere, possibly except for a       
subspace of a small dimension. 
%%
%%If Algorithm~\ref{alg:g2g} is used as a preconditioner,
%
In the context of preconditioning, this translates into the fact that
the preconditioned matrix has spectrum clustered around $1$ and $-1$
with a few outliers generated by the amplification of the smooth eigenmodes.   
If the shift is sufficiently small, the number of such outliers is not large, 
%the increase of
%outliers in the clustered spectrum of the preconditioned matrix, which 
%results in a larger number of 
which only slightly delays the convergence of the outer PMINRES iterations
and does not significantly affect the efficiency of the overall scheme. 
\subsection{Algorithm~\ref{alg:g2g} with $B = p_m (L - c^2 I)$}\label{subsec:Bpoly}
The analysis of the previous subsection suggests that the quality 
of Algorithm~\ref{alg:g2g} with $B=L$ may deteriorate as 
%the quantity 
$ch$ increases. This result is not surprising, since for larger $ch$
the relation $L \approx |L - c^2 I|$ becomes no longer meaningful.  
Below we introduce a different approach for approximating the fine grid
absolute value. In particular, we consider constructing \textit{polynomial approximations}
$B = p_m (L - c^2 I)$, where $p_m(\lambda)$ is a polynomial of degree
at most $m>0$, such that $p_m(L - c^2 I) \approx |L - c^2 I|$.

%The fact that Algorithm~\ref{alg:avp-gmg} requires sufficiently large sizes 
%of the coarsest grids is not surprising.
%Laplacians $L_l$ constitute deteriorating approximations of the 
%absolute values $|L_l - c^2I_l|$ in~(\ref{eqn:mg-pre}),~(\ref{eqn:mg-cgc-2}),
%and~(\ref{eqn:mg-post}) as $l$ decreases. Thus, 
%for sufficiently coarse grids, the relation 
%$L_l \approx |L_l - c^2 I_l|$ is no longer meaningful, which
%motivates the immediate coarse grid solve, resulting in the requirement 
%on the coarsest grids to be fine enough.  
%%
%As has been shown in the previous subsection, this is not a problem
%for small shifts. For sufficiently large shifts, however, the sizes of the 
%coarsest grids become so large that the whole computation is dominated
%by the coarse grid solves~(\ref{eqn:mg-cgc-1}).
%%
%The question remains if Algorithm~\ref{alg:avp-gmg}
%can be extended to deliver an efficient absolute value preconditioner 
%for systems with a higher level of indefiniteness. Below, we 
%show that this can be done by using \textit{polynomial approximations} 
%of absolute value operators on coarser grids.  

Let us first refer to the ideal particular case, where
$p_m (L - c^2 I) = |L - c^2 I|$. This can happen, e.g.,
if $p_m(\lambda)$ is an interpolating polynomial of $f(\lambda) = |\lambda|$
on the spectrum of $L - c^2 I$, $m = n-1$.
In such a situation, Algorithm~\ref{alg:g2g}
with $B = p_m(L - c^2 I)$ results in the following transformation 
of the initial error: % $e_0^{\mbox{{\tiny AV}}} = |L - c^2 I|^{-1} r$: 
\begin{equation}\label{eqn:eAVpoly}
e^{\mbox{{\tiny AV}}} =  \bar S_2^{\nu} \bar K \bar S_1^{\nu}  e_0^{\mbox{{\tiny AV}}},
\end{equation}
where $\bar S_1 = I - M^{-1} |L-c^2 I|$ and $\bar S_2 =  I - M^{-*} |L - c^2 I|$ are pre- and postsmoothing
operators, and $\bar K = I - P|L_H - c^2 I_H|^{-1} R |L-c^2 I|$ corresponds to the coarse grid
correction step. 
%, and $e^{\mbox{{\tiny AV}}} = |L - c^2 I|^{-1} r - w^{post}$.  
The associated error propagation operator is further denoted by $\bar G$, 
%i.e.,  
%If $p_m(L - c^2 I) = |L - c^2 I|$, then 
%Algorithm~\ref{alg:g2g}
%with $B = p_m(L - c^2 I)$ results in the following transformation 
%of the initial error $e_0^{\mbox{{\tiny AV}}} = |L - c^2 I|^{-1} r$: 
%\begin{equation}\label{eqn:eAVpoly}
%e^{\mbox{{\tiny AV}}} =  \bar S_2^{\nu} \bar K \bar S_1^{\nu}  e_0^{\mbox{{\tiny AV}}}.
%\end{equation}
\begin{equation}\label{eqn:Gpoly}
\bar G = \bar S_2^{\nu} \bar K \bar S_1^{\nu}.
\end{equation}  

For the purpose of clarity, we again consider the 1D 
counterpart~(\ref{eqn:helmholtz_bvp1D}) of the model problem. 
As a smoother, we choose Richardson's iteration 
with respect to absolute value system~(\ref{eqn:av_sys}), i.e.,
$\bar S_1 = \bar S_2 = I - \tau |L - c^2 I|$.
%, where $\tau$ is an 
%iteration parameter. 
%
It is important to note here that the eigenvalues 
$|t_j|$ of the absolute value operator are, in general, 
no longer ascendingly ordered with respect to $j$
%corresponds to the ascending
%order 
as is the case for $\theta_j$'s and $t_j$'s. Moreover,
in contrast to $L$ and $L - c^2 I$,
%for sufficiently large shifts $c^2$,  
the top part of the spectrum of $|L - c^2 I|$ may be associated with both 
smooth and oscillatory eigenmodes. 
%More precisely, if $c^2$ is sufficiently large, there exists $N^* \leq N$, 
%such that $|\lambda_j| = |\mu_j - c^2| > |\mu_{N+1} - c^2| = |\lambda_{N+1}|$
%for $j = 1,\ldots,N^*$, whereas the associated eigenmodes $v_j$ are smooth.
%
In particular, this means that Richardson's iteration
may fail to properly eliminate the oscillatory components of the error, 
which is an undesirable outcome of the smoothing procedure. 
To avoid this, we require that $|t_1| < t_{N+1}$. It is easy
to verify that the latter condition is fulfilled if 
%which is guaranteed by the condition
\begin{equation}\label{eqn:ch_coarse}
ch < 1.
\end{equation}  
Note that~(\ref{eqn:ch_coarse}) automatically holds if discretization
rule~(\ref{eqn:thumb}) is enforced. 
%Additionally, we remark that 
Repeating the above
argument for the 2D case also leads to~(\ref{eqn:ch_coarse}).

Let the restriction and prolongation operators $R$ and
$P$ be the same as in the previous subsection. 
%exactly the same as in Subsection~\ref{subsec:BLaplacian}. 
Similar to~(\ref{eqn:Kv}), 
%we use~(\ref{eqn:Rv}) and~(\ref{eqn:Pv}) to 
we obtain 
an explicit expression for the action of the coarse grid correction operator
$\bar K$ on eigenvectors $v_j$: 
\begin{equation}\label{eqn:Kv-poly}
\bar K v_j = \left\{ 
\begin{array}{ll}
\displaystyle \left( 1 - c_j^{4}  \frac{| t_j |}{| t_j^H |}  \right) v_j + s_j^2 c_j^2  \frac{|t_j|}{|t_j^{H}|}  v_{n+1-j}, & j = 1, \ldots ,N, \\
\displaystyle v_j, & j = N+1, \\
\displaystyle \left( 1 - c_j^4 \frac{|t_j|}{|t^H_{n+1-j}|} \right)v_j  + s_j^2 c_j^2 \frac{|t_j|}{|t_{n+1-j}^H|} v^{H}_{n+1 -j}, 
& j = N +2, \ldots ,n.
\end{array}
\right.
\end{equation} 

The following theorem is the analogue of Theorem~\ref{thm:G}. 
%for the 
%ideal case, where $B = p_m(L-c^2 I) = |L - c^2 I|$ 
%We can now obtain the explicit expression for the images of eigenmodes $v_j$ under the action of the 
%error propogation operator $\bar G$ in~(\ref{eqn:Gpoly}).
%Note that~(\ref{eqn:Kv-poly}) is the same, up to the occuring absolute values, as the expression
%for the coarse-grid correction operator in~\cite[Theorem 2.1]{Elman.Ernst.OLeary:01} of the standard two-grid scheme
%applied to~(\ref{eqn:helmholtz_fd}).

\begin{theorem}\label{thm:G-poly}
The error propagation operator $\bar G$ in~(\ref{eqn:Gpoly}) acts on the 
eigenvectors $v_j$ of the 1D Laplacian as follows: 
\begin{equation}\label{eqn:Gv-poly}
\bar G v_j = \left\{ 
\begin{array}{ll}
\displaystyle  \bar g_j^{(11)} v_j + \bar g_j^{(12)} v_{n+1-j}, & j = 1,\ldots,N, \\
\displaystyle  \bar g_j v_j, & j = N+1, \\
\displaystyle \bar g_j^{(21)} v_j   + \bar g_j^{(22)} v_{n+1 -j}, & j = N +2, \ldots,n,
\end{array}
\right.
\end{equation}  
where
\begin{eqnarray}
\label{eqn:g11-poly} 
\bar g_j^{(11)} & = &(1 - \tau |t_j|)^{2\nu} \left( 1 - c_j^4 \frac{|t_j|}{|t_j^H|}  \right), \\ 
\label{eqn:g12-poly}
\bar g_j^{(12)} & = & (1 - \tau |t_j|)^{\nu} s_j^2 c_j^2 \frac{|t_j|}{|t_j^{H}|} (1 - \tau |t_{n+1-j}|)^{\nu},  \\
\label{eqn:g-poly}
\bar g_j & =  & (1-\tau |t_j|)^{2 \nu}, \\
\label{eqn:g21-poly}
\bar g_j^{(21)} & = &  (1 - \tau |t_j|)^{2\nu} \left( 1 - c_j^4 \frac{|t_j|}{|t_{n+1-j}^H|}\right), \\ 
\label{eqn:g22-poly}
\bar g_j^{(22)} & = & (1 - \tau |t_j|)^{\nu} s_j^2 c_j^2 \frac{|t_j|}{|t_{n+1-j}^{H}|} (1 - \tau |t_{n+1-j}|)^{\nu}. 
\end{eqnarray}
\end{theorem}

We conclude from
Theorem~\ref{thm:G-poly} 
%allows showing 
that in the ideal case where $p_m(L - c^2 I) = |L - c^2 I|$, 
Algorithm~\ref{alg:g2g} with $B = p_m(L - c^2 I)$ and a proper choice of $\tau$ and $\nu$ 
reduces the error of system~(\ref{eqn:av_sys})
in the directions of all eigenvectors $v_j$, possibly except for a few that 
% In the directions of a few eigenmodes, however,
%the error may be amplified. These eigenmodes 
correspond to $\theta_j$ close to the shift $c^2$. 
Unlike in the case of Algorithm~\ref{alg:g2g} with $B = L$,
as $ch$ grows, no amplified error components appear in the directions of eigenvectors associated 
with $\theta_j$ distant from $c^2$ on the left. 
This suggests that Algorithm~\ref{alg:g2g} with 
$B = p_m(L - c^2 I) = |L - c^2 I|$ provides a more accurate solve for~(\ref{eqn:av_sys}) 
with larger $ch$.

To see this, let us first assume that $\tau = h^2/(3 - c^2 h^2)$.
%Let us assume that the smoothing parameter $\tau$ is chosen to ensure
%that $\left| 1 - \tau|\lambda_j| \right|$ is small for $j>N$.
Since~(\ref{eqn:ch_coarse}) implies that $|t_j| = t_j$ for $j>N$, 
%the eigenvalues of $L - c^2 I$
%and $|L-c^2 I|$ associated with the oscillatory eigenmodes $v_j$ are the same,
%i.e., $|\lambda_j| = \lambda_j$ for $j>N$, 
%this can be done, e.g.,
%by choosing 
%%\begin{equation}\label{eqn:tau}
%$\tau = h^2/(3 - c^2 h^2)$.
%%\end{equation}
this choice is known to give the smallest uniform bound on $|1 - \tau |t_j||$
corresponding to the oscillatory eigenmodes $v_j$, which is $|1 - \tau |t_j|| < 1/(3-ch) < 1/2$
with the last inequality resulting from~(\ref{eqn:ch_coarse}).
%;~e.g.,~\cite{Elman.Ernst.OLeary:01}. 
%corresponds to the optimal damping of the oscillatory error components 
%and results in the 
Hence, coefficients~(\ref{eqn:g12-poly})--(\ref{eqn:g22-poly}) can be reduced 
within a reasonably small number $\nu$ of smoothing steps. 
%Note that if $c^2 = 0$ then $\tau = h^2/3$, as in the previous subsection. 

Next, we note that~(\ref{eqn:g11-poly}), which is not substantially affected by smoothing, 
%it is readily seen that reduction coefficient~(\ref{eqn:g11-poly})
can be large if $c^2$ is close to $\theta_j^H$, i.e., if $t_j^H \approx 0$.
At the same time, 
%using $\theta_j^H = c_j^2 \theta_j$, 
we can write~(\ref{eqn:g11-poly})~as
\[
g_j^{(11)} = (1 - \tau |t_j|)^{2\nu} \left(1 - c_j^2 \left| 1 + \frac{c^2 s_j^2}{t_j^H}\right| \right),
\]
which shows that $|g_j^{(11)}|$ 
%is large if $\lambda_j^H \approx 0$. 
%however, gets closer to
approaches $(1 - \tau |t_j|)^{2\nu} s_j^2 < 1/2$ 
as $|t_j^H|$ increases, i.e.,
smooth error components associated with $\theta_j$ away from $c^2$ are well damped. 
% by the coarse grid correction. 
%$j$ corresponding to larger values of $|\lambda_j^H|$. 
%
%We note that if $\nu$ is too small and $\lambda_j^H \approx 0$, 
%the corresponding coefficients~(\ref{eqn:g12-poly}), (\ref{eqn:g21-poly}), and~(\ref{eqn:g22-poly}) may be large, i.e., 
%the error may grow in the direction of the eigenmode complementary to $v_j$. 
%

Thus, if used as a preconditioner, Algorithm~\ref{alg:g2g} with $B = p_m(L - c^2 I) = |L - c^2 I|$ 
aims at clustering the spectrum of the preconditioned matrix around $1$ and $-1$, with a few
possible outliers that result from the amplification of the smooth eigenmodes associated with
$\theta_j$ close to $c^2$.    
Unlike in the case where $B = L$, the increase of $ch$ does not additionally amplify the smooth 
error components distant from $c^2$ on the left. Therefore,
Algorithm~\ref{alg:g2g} with $B = p_m(L - c^2 I) = |L - c^2 I|$ can be expected to provide 
a more accurate preconditioner for larger shifts.

%the number of the amplified error components does not increase as $ch$
%grows.  

%As already discussed, in the context of preconditioning, such a deficiency 
%results in a slight delay in the convergence of the optimal, and inexpensive, outer PMINRES iterations.   

%%
%This observation suggests that for larger $ch$ 
%Algorithm~\ref{alg:g2g} with $B = p_m(L - c^2 I) = |L - c^2 I|$ 
%results in a more accurate preconditioner.
%%This makes the option 
%%$B = p_m(L - c^2 I)$ preferable for larger $ch$. 
%%

Although our analysis targets the ideal but barely feasible case 
where $p_m(L - c^2 I) = |L - c^2 I|$, it motivates 
the use of \textit{polynomial approximations} $p_m(L - c^2 I) \approx |L - c^2 I|$
and provides a theoretical insight into the superior behavior of such an option
for larger $ch$. 
In the rest of this subsection we describe a method for constructing such polynomial
approximations. 
% to the absolute value operators. 
Our approach is based on the 
finding that the problem is easily reduced to constructing
\textit{polynomial filters}. 
%, which is reasonably well studied.       

We start by introducing the step function
\[
h_{\alpha}(\lambda) = \left\{ 
\begin{array}{ll}
\displaystyle  1, & \lambda \geq \alpha, \\
\displaystyle  0, & \lambda  < \alpha; 
\end{array}
\right.
\]
where $\alpha$ is a real number, 
and noting that $\mbox{sign}(\lambda) = 2h_0(\lambda) - 1$, so that
%
%By~(\ref{eqn:absval}) with $A \equiv L - c^2 I$, 
%this allows writing the expression for the absolute value operator
%in the form
%
\begin{equation}\label{eqn:avp_step}
|L - c^2 I|  =   \left( 2h_0(L - c^2 I) - I \right)\left(L - c^2 I \right). 
\end{equation}
Here $h_0 (L - c^2 I) = V h_0(\Lambda) V^*$, where 
$V$ is the matrix of eigenvectors of $L - c^2 I$ and 
$h_0(\Lambda) = \mbox{diag}\{0,\ldots,0,1,\ldots,1\}$ 
is obtained by applying the step function $h_0(\lambda)$ to the diagonal entries of the matrix $\Lambda$ of 
the associated eigenvalues. Clearly the number of zeros on the diagonal of $h_0 (\Lambda)$
equals the number of negative eigenvalues of $L - c^2 I$.

Let $q_{m-1}(\lambda)$ be a polynomial of degree at most $(m-1)$, such that
$q_{m-1}(\lambda)$ approximates $h_0(\lambda)$ on the interval $[a,b]$, where
$a$ and $b$ are the lower and upper bounds on the spectrum of $L - c^2 I$, 
respectively.
In order to construct an approximation $p_m(L - c^2 I)$ of $|L - c^2 I|$, we replace the step function 
$h_0(L - c^2 I)$ in~(\ref{eqn:avp_step}) by the polynomial  $q_{m-1}(L - c^2 I)$. Thus,
%so that   
%at most $d$, such that $\phi_d(\lambda)$ approximates $h_0(\lambda)$. This results in the 
%polynomial approximation 
%
\begin{equation}\label{eqn:avp_poly}
|L  - c^2 I| \approx p_m(L - c^2 I) = \left( 2 q_{m-1} (L - c^2 I) - I \right)\left(L - c^2 I\right). 
\end{equation}

The matrix $L - c^2 I$ is readily available on the fine grid. Therefore,
we have reduced the problem of evaluating the polynomial approximation
$p_m$ of the absolute value operator to constructing a polynomial
$q_{m-1}$ that approximates the step function~$h_0$.
More specifically, since Algorithm~\ref{alg:g2g} can be implemented without the explicit
knowledge of the matrix $B$, i.e., $B$ can be accessed only through its action on a vector,
we need to construct approximations of the form $q_{m-1}(L - c^2 I) v$ 
to $h_{0}(L - c^2 I) v$, where $v$ is a given vector.

The task of constructing $q_{m-1}(L - c^2 I) v \approx h_0(L - c^2 I) v$
represents an instance of \textit{polynomial filtering},
which is well known;  
see, e.g.,~\cite{Erhel.Guyomarc.Saad:01, Saad:06, Zhou.Saad.Tiago.Chelikowski:06}.
%This procedure represents an 
%instance of the \textit{polynomial filtering}
%which is reasonably well known and studied in literature;  
%see, e.g.,~\cite{Erhel.Guyomarc.Saad:01, Saad:06, Zhou.Saad.Tiago.Chelikowski:06}.
In this context, due to the property of filtering out certain undesirable 
eigencomponents, the step function $h_{0}$ is called 
a \textit{filter function}. 
The approximating polynomial $q_{m-1}$ is referred to as
a \textit{polynomial filter}.

State-of-the-art polynomial filtering techniques such as~\cite{Saad:06} 
would first replace the discontinuous step function $h_{0}(\lambda)$ 
by a smooth approximation on $[a,b]$ and then approximate the latter by a polynomial 
in the least-squares sense. 
In this paper, we follow a simpler approach based on the 
direct approximation of $h_0(\lambda)$ using  
\textit{Chebyshev polynomials}~\cite{Powell:81, Rivlin:81}.
The constructed polynomial $q_{m-1}$ allows defining 
$q_{m-1}(L - c^2 I) v \approx h_0(L - c^2 I) v$ and hence 
$p_{m}(L - c^2 I) v \approx |L - c^2 I| v$.
\iftoggle{report}{
Below
we describe the entire procedure.  

Let 
\begin{equation}\label{eqn:ksi}
\xi = \left(\frac{2}{b-a}\right)\lambda - \frac{b+a}{b-a}
\end{equation}
be a linear change of variable which maps $\lambda \in [a,b]$ to $\xi \in [-1,1]$.  
Then note that
\[
h_0 (\lambda) = h_\alpha (\xi), \ \alpha = -\frac{b+a}{b-a}.
\]

Let $T_i(\xi) = \mbox{cos}( i \, \mbox{arccos} (\xi) )$ be 
the Chebyshev polynomials of the first kind, 
%where 
$\xi \in [-1,1]$. 
%and $i = 0,\ldots,d$. 
%
Then the Chebyshev least-squares approximation $\bar q_{m-1}(\xi)$ of the 
step function $h_\alpha(\xi)$ is of the form
\begin{equation}\label{eqn:exp_cheb}
\bar q_{m-1}(\xi) = \sum_{i=0}^d \gamma_i T_i (\xi), \ \xi \in [-1,1],
\end{equation}
where the expansion coefficients are given by
\[
\gamma_i = \displaystyle  \frac{1}{n_i} \int_{-1}^{1} \frac{1}{\sqrt{1-\xi^2}} h_{\alpha} (\xi) T_i (\xi) d\xi, \ i = 0,\ldots,m-1;
%\end{array}
%\right.
\]
with $n_i = \displaystyle \int_{-1}^{1} \frac{1}{\sqrt{1 - \xi^2}} T^2_i(\xi) d \xi$.
%and $\alpha \in [-1,1]$.
%
Calculating the above integrals gives the exact expressions for $\gamma_i$,
\begin{equation}\label{eqn:coeff_cheb}
\gamma_i = \left\{ 
\begin{array}{ll}
\displaystyle  \frac{1}{\pi} \left( \mbox{arccos}(\alpha) \right), & i = 0, \\
 & \\ 
\displaystyle  \frac{2}{\pi} \left( \frac{\mbox{sin}(i \, \mbox{arccos}(\alpha))}{i} \right), & i \geq 1.
\end{array}
\right.
\end{equation}

Since $h_0(\lambda) = h_\alpha(\xi)$, we define the polynomial approximation
$q_{m-1} (\lambda)$ to $h_0 (\lambda)$ as 
%we define the desired polynomial approximation of $h_0 (\lambda)$ on $[a,b]$ as 
%\[
$
q_{m-1} (\lambda) \equiv \bar q_{m-1}(\xi), 
$
%\]
where $\xi$ is given by~(\ref{eqn:ksi}),
$\alpha = -(b+a)/(b-a)$, and $\bar q_{m-1}(\xi)$
is constructed according to~(\ref{eqn:exp_cheb}),~(\ref{eqn:coeff_cheb}).
Thus, by~(\ref{eqn:avp_poly}), we get
\[
|L  - c^2 I|v \approx p_m(L - c^2 I)v = 2 \sum_{i=0}^{m-1} \gamma_i T_i(C) t - t, \ t = (L - c^2 I) v, 
\]
where
$\displaystyle C = \frac{2}{b-a} \left(L - c^2 I\right) - \frac{b+a}{b-a}I$.
Since the Chebyshev polynomials are generated using the three-term recurrent relation~\cite{Powell:81, Rivlin:81}
\[
T_i (\xi) = 2 \xi T_{i-1} (\xi) - T_{i-2} (\xi), \ T_1 (\xi) = \xi, \ T_0(\xi) = 1, \ i = 2,3,\ldots,
\]
the procedure for constructing the desired vectors $w = p_m(L - c^2 I) v \approx |L - c^2 I| v$
can be summarized by the following algorithm.

\vspace{0.1in}
\begin{algorithm}[Polynomial approximation of the absolute value]\label{alg:av}
Input: $v$, $m$, $[a, b] \supset \Lambda(L-c^2 I)$. Output: $w = p_m(L-c^2I) v \approx |L-c^2 I|v$.
\begin{enumerate}
\item \emph{Set} $v \leftarrow (L - c^2 I)v$. 
\item \emph{Set} $\displaystyle C \leftarrow \frac{2}{b-a} \left(L - c^2 I\right) - \frac{b+a}{b-a} I$. \emph{Set} $\alpha \leftarrow - (b+a)/(b-a)$.
\item \emph{Set} $v_0 \leftarrow v$, $v_1 \leftarrow Cv$, $w_1 \leftarrow \gamma_0 v_0 + \gamma_1 v_1$.  
\emph{Throughout, compute $\gamma_i$ by~(\ref{eqn:coeff_cheb})}.
\item  \emph{\textbf{For}} $i = 2, \ldots, m-1$ \emph{\textbf{do}} 
  \item \hspace{0.2in} 
         $v_i \leftarrow 2Cv_{i-1} - v_{i-2}$
  \item  \hspace{0.2in} 
        $w_i \leftarrow w_{i-1} + \gamma_i v_i$
\item \emph{\textbf{EndFor}}
\item \emph{Return} $w \leftarrow 2 w_{m-1} - v$.
\end{enumerate}
\end{algorithm}
\vspace{0.1in}

Algorithm~\ref{alg:av} provides means to replace a matrix-vector product 
with the unavailable $|L - c^2 I|$ by a procedure that involves a few 
multiplications with the shifted Laplacian $L - c^2 I$. As we further show,
the degree $m$ of the approximating polynomial can be kept reasonably low.
Moreover, in the MG framework discussed in the 
next subsection, the algorithm has to be invoked only on sufficiently coarse grids. 
} % end \iftoggle
{
% This is the submission (shorter) version
Thus, the entire procedure provides means to replace a matrix-vector product 
with the unavailable $|L - c^2 I|$ by, essentially, a few 
multiplications with $L - c^2 I$. As we further show,
the degree $m$ of the approximating polynomial can be kept reasonably low.
Moreover, in the MG framework discussed in the 
next subsection, the algorithm has to be invoked only on sufficiently coarse grids. 
}

\subsection{The MG AV preconditioner}\label{subsec:mg-avp}

Now let us consider a hierarchy of $s + 1$ grids numbered by $l = s, s-1, \ldots, 0$ with the corresponding mesh sizes
$\left\{h_l\right\}$ in decreasing order ($h_s = h$ corresponds to the finest grid, and $h_0$ to the coarsest).
For each level $l$ we define the discretization $L_l - c^2 I_l$ of the differential operator in 
(\ref{eqn:helmholtz_bvp}), where $L_l$ is the Laplacian on grid $l$, and 
$I_l$ is the identity of the same size.

In order to extend the two-grid AV preconditioner given by Algorithm~\ref{alg:g2g} to the  
\textit{multigrid}, 
instead of inverting the absolute value $\left|L_H - c^2 I_H \right|$ 
in (\ref{eqn:cgc-1}), we recursively apply 
the algorithm to the restricted vector $R(r - B w^{pre})$. 
This pattern is then followed in the V-cycle fashion
%~\cite{Briggs.Henson.McCormick:00, Trottenberg.Oosterlee.Schuller:01}
on all levels, with the inversion of the absolute value of the shifted Laplacian on the coarsest grid.
The matrix $B$ on level $l$ is denoted by $B_l$. 
Each $B_l$ is assumed to be SPD and is expected to approximate $|L_l - c^2 I_l|$. 
%The described approach can be viewed as replacing $w_H$ in (\ref{eqn:cgc-1}) by its approximation, i.e., 
%constructing $w_H  \approx   \left|L_H - c^2 I_H \right|^{-1} R \left(r - L w^{pre}\right)$. 

In the previous subsections we have considered two choices of $B$ for the two-grid preconditioner
in Algorithm~\ref{alg:g2g}. 
In the MG framework, these choices give $B_l = L_l$ and $B_l = p_{m_l}(L_l - c^2 I_l)$,
where $p_{m_l}$ is a polynomial of degree at most $m_l$ on level $l$. 
%It is assumed that all $p_{m_l}(L_l - c^2 I_l)$ are SPD and expected that $B_l \approx |L_l - c^2 I_l|$.

The advantage of the first option, $B_l = L_l$, is that it can be easily constructed
and the application of $B_l$ to a vector is inexpensive even if the size of the operator
is very large.   
According to our analysis for the 1D model problem in subsection~\ref{subsec:BLaplacian},
the approach is suitable for $ch_l$ sufficiently small. Typically this is a case
for $l$ corresponding to finer grids. However, $c h_l$ increases with every new level. 
This may result in the deterioration of accuracy of the overall MG preconditioning scheme, 
unless the size of the coarsest level is kept sufficiently large.

The situation is different for the second option $B_l = p_{m_l} (L_l - c^2 I_l)$.
In this case, applications of $B_l$ may be expensive on finer grids because they
require a sequence of matrix-vector multiplications with large shifted
Laplacian operators. However, on coarser levels, i.e., for larger $ch_l$, this
is not restrictive because the involved operators are
significantly decreased in size compared to the finest level.     
Additionally, as suggested by the analysis in subsection~\ref{subsec:Bpoly},
if $p_{m_l} (L_l - c^2 I_l)$ represent reasonable approximations of 
$|L_l - c^2 I_l|$ on levels $l$, one can expect a higher
accuracy of the whole preconditioning scheme compared
to the choice $B_l = L_l$.   

Our idea is to combine the two options.
Let $\delta \in (0,1)$ be a ``switching'' parameter, where for finer
grids $ch_l < \delta$. We choose %$B_l = L_l$. Otherwise,
%we set $B_l = p_{m_l} (L_l - c^2 I_l)$. Thus, 
\begin{equation}\label{eqn:Bl}
B_{l} = \left\{ 
\begin{array}{ll}
\displaystyle  L_l, & ch_l < \delta, \\
\displaystyle  p_{m_l}(L_l - c^2 I_l), & ch_l \geq \delta. 
\end{array}
\right.
\end{equation}
\iftoggle{report}{
The polynomials $p_{m_l}(L_l - c^2 I_l)$ are accessed
through their action on a vector and are constructed using Algorithm~\ref{alg:av}
with $L - c^2 I \equiv L_l - c^2 I_l$ and $m \equiv m_l$. 
}
{
% Short version
The polynomials $p_{m_l}(L_l - c^2 I_l)$ are accessed
through their action on a vector. 
}

Summarizing our discussion, if started from the finest grid $l = s$, 
the following scheme gives the multilevel extension of the two-grid AV preconditioner defined
by Algorithm~\ref{alg:g2g}.  
The subscript $l$ is introduced to match quantities 
to the corresponding grid. We assume that the parameters 
$\delta$, $m_l$, $\nu_l$, and the smoothers $M_l$ are pre-specified.

%Finally, we state the new MG absolute value preconditioning
%procedure in Algorithm~\ref{alg:avp-gmg-poly}. 
%The scheme is a modification of Algorithm~\ref{alg:avp-gmg} 
%which, as we will see in the next section, 
%allows handling larger shifts in system~(\ref{eqn:helmholtz_fd}) by taking advantage
%of the polynomial approximations of the absolute value operators on coarser grids. 
\vspace{0.1in}
\begin{algorithm}[AV-MG($r_l$): the MG AV preconditioner]\label{alg:avp-gmg}
\vspace{.1in}
Input $r_l$. Output~$w_l$.
%
%Output: w
\begin{enumerate}
\item Set $B_l$ by~(\ref{eqn:Bl}). % Use Algorithm~\ref{alg:av} for matrix-vector products with $B_l = p_{m_l} (L_l - c^2 I_l)$. 
	\item \emph{Presmoothing}. Apply $\nu_l$ smoothing steps, $\nu_l \geq 1$:
\begin{equation}\label{eqn:mg-pre}
w_l^{(i+1)} = w_l^{(i)} + M_l^{-1} (r_l - B_l w_l^{(i)}),  \ i = 0,\ldots,\nu_l - 1, \ w_l^{(0)} = 0,
\end{equation}
where $M_l$ defines a smoother on level~$l$. 
Set $w_l^{pre} = w_l^{(\nu)}$. 
	\item \emph{Coarse grid correction}. Restrict ($R_{l-1}$) $r_l - B_l w_l^{pre}$ to the grid $l-1$,
        recursively apply AV-MG, and prolongate ($P_l$) back to the fine grid.
         This delivers the coarse grid correction added to $w_l^{pre}$:
\begin{equation}\label{eqn:mg-cgc-1}
w_{l-1} = \left\{ 
\begin{array}{ll}
\displaystyle  \left|L_0 - c^2 I_0 \right|^{-1} R_0 \left(r_1 - B_1 w_1^{pre}\right), & l = 1, \\
\displaystyle  \mbox{AV-MG}\left(R_{l-1} \left(r_l - B_l w_l^{pre}\right) \right), & l > 1; 
\end{array}
\right.
\end{equation}
%	If $l = 1$, then %multiply the restricted vector by the inverted coarse-level absolute value $\left|L_0 - c^2 I_0\right|$,
%	\begin{equation}\label{eqn:mg-cgc-1}
%	w_0  =   \left|L_0 - c^2 I_0 \right|^{-1} R_0 \left(r_1 - B_1 w_1^{pre}\right), \ \mbox{if} \ l=1.
%	\end{equation}
%	Otherwise, recursively apply AV-MG: 
%	\begin{equation}\label{eqn:mg-cgc-2}
%	w_{l-1}  =   \mbox{AV-MG}\left(R_{l-1} \left(r_l - B_l w_l^{pre}\right) \right), \ \mbox{if} \ l>1.
%	\end{equation}
%	Prolongate the result back to the fine grid and
%	This delivers the coarse grid correction, which is 
%add to $w_l^{pre}$ to obtain the corrected vector $w_l^{cgc}$:
\begin{equation}\label{eqn:mg-cgc-2}
  w_l^{cgc}  =  w_l^{pre} + P_l w_{l-1}.
\end{equation}
	\item \emph{Postsmoothing}. Apply $\nu_l$ smoothing steps:
\begin{equation}\label{eqn:mg-post}
w_l^{(i+1)} = w_l^{(i)} + M_l^{-*} (r_l - B_l w_l^{(i)}),  \ i = 0,\ldots,\nu_l - 1, \ w_l^{(0)} = w_l^{cgc},
\end{equation}
where $M_l$ and $\nu_l$ are the same as in step 2. 
Return $w_l = w_l^{post} =  w_l^{(\nu_l)}$.
%Set $w_l^{post} = w_l^{(\nu)}$.   
 \end{enumerate}
\end{algorithm}

The described MG AV preconditioner implicitly
constructs a mapping denoted by $r \mapsto w = T_{mg} r$, where the operator $T = T_{mg}$ has 
the following structure:
\begin{equation}\label{eqn:mg_struct}
T_{mg}  =  \left(I - M^{-*} B \right)^{\nu} P T_{mg}^{(s-1)} R \left(I -  B M^{-1} \right)^{\nu} + F,
\end{equation}
with $F$ as in (\ref{eqn:2grid_struct}) and $T_{mg}^{(s-1)}$ defined according to the recursion
%In other words, $T^{(m-1)}_{mg}$ is defined as follows:
\begin{equation}\label{eqn:vcyc}
\begin{array}{lll}
%\begin{equation}
T^{(l)}_{mg} & = & \left(I_{l} - M^{-*}_l B_l \right)^{\nu_l} P_{l}  T^{(l-1)}_{mg} R_{l-1} \left(I_l -  B_l M_l^{-1} \right)^{\nu_l} + F_l, \\
%%l & = & 1,\ldots,m-1, \\
T^{(0)}_{mg} & = & \left| L_0 - c^2 I_0 \right|^{-1}, \ l  =  1,\ldots,s-1, 
\end{array}
\end{equation} 
where $F_l = B_l^{-1} - \left(I_l - M_l^{-*} B_l \right)^{\nu_l} B_l^{-1}\left(I_l -  B_l M_l^{-1}\right)^{\nu_l}$.

%In (\ref{eqn:mg_struct}), we skip 
%the subscript in the notation for the quantities associated with the finest level $l=s$.
The structure of the multilevel preconditioner $T = T_{mg}$ in (\ref{eqn:mg_struct}) is the same as that of the 
two-grid preconditioner $T = T_{tg}$ in (\ref{eqn:2grid_struct}), with $\left| L_H - c^2 I_H \right|^{-1}$
replaced by the recursively defined operator $T^{(m-1)}_{mg}$ in (\ref{eqn:vcyc}).
Thus, the symmetry and positive definiteness of $T = T_{mg}$ 
follows from the same property of the two-grid operator through relations~(\ref{eqn:vcyc}),
provided that $P_l = \alpha_l R_{l-1}^*$ and the spectral radii of $I_l - M_l^{-1} B_l$
and $I_l - M_l^{-*} B_l$ are less than $1$ throughout the coarser levels. 
%If the assumptions on the fine-grid operators $M$, $M^*$, $R$ and $P$, are sufficient to ensure that the two-grid 
%preconditioner in (\ref{eqn:2grid_struct}) is SPD, it 
%remain valid throughout the coarser levels, i.e., $P_l = \alpha R_{l-1}^*$, and the spectral radii of $I_l - M_l^{-1} B_l$
%and $I_l - M_l^{-*} B_l$ are less than $1$, $l = 1,\ldots,s-1$,  
%then the symmetry and positive definiteness of the MG preconditioner $T = T_{mg}$ in~(\ref{eqn:mg_struct})
%can be obtained from the same property of the two-grid operator  
%through relations~(\ref{eqn:vcyc}).
%
We remark that preconditioner~(\ref{eqn:mg_struct})--(\ref{eqn:vcyc}) is non-variable, i.e., it 
preserves the global optimality of PMINRES.  
%All occurring polynomials $B_l = p_{m_l}(L_l - c^2 I_l)$ are assumed to be SPD. 

The simplest possible approach for computing $w_0$ in~(\ref{eqn:mg-cgc-1})
is to explicitly construct $|L_0 - c^2 I_0|^{-1}$ through the full 
eigendecomposition of the coarse-level Laplacian, and then apply it to $R_0(r_1 - B_1 w_1^{pre})$. 
An alternative approach
%the matrix  $V_0$ of eigenvectors associated with the negative eigenvalues of $L_0 - c^2 I_0$,
%contained in the corresponding diagonal matrix $\Lambda_0$,%~(\ref{eqn:av_lowrank}),
%and determines 
is to determine $w_0$ as a solution of the linear system
%\begin{equation}\label{eqn:cgc-system}
$(L_0 - c^2 I_0 + 2 V_0 |\Lambda_0| V^*_0) w_0 = R_0(r_1 - B_1 w_1^{pre})$, 
%\end{equation}
where $V_0$ is the matrix of eigenvectors associated with the negative eigenvalues of $L_0 - c^2 I_0$
contained in the corresponding diagonal matrix $\Lambda_0$. %~(\ref{eqn:av_lowrank}),
In the latter case, the full eigendecomposition of $L_0$ is replaced 
by the \textit{partial} eigendecomposition targeting negative eigenpairs,
followed by a linear solve.   
\iftoggle{report}{
To solve the linear system \textit{approximately},
additional gains may be obtained by applying the Woodbury formula~\cite[p.50]{Golub.VanLoan:96}
to the matrix $(L_0 - c^2 I_0 + 2 V_0 |\Lambda_0| V^*_0)^{-1}$,
provided that systems $C z = r$ can be efficiently solved 
for some $C \approx L_0 - c^2 I_0$. 
In this work, we rely only on exact coarse grid solves.
}
%
%The sizes of the coarsest-level problems are discussed in the next section. 

Since we use Richardson's iteration with respect to
$p_{m_l} (L_l - c^2 I_l)$ as a smoother on coarser grids, as 
motivated by the discussion in subsection~\ref{subsec:Bpoly}, 
the guidance for the choice of the coarsest grid is given by 
condition~(\ref{eqn:ch_coarse}).
More specifically, in the context of the standard coarsening procedure ($h_{l-1} = 2 h_l$),
we select hierarchies of grids satisfying $ch_l < 1$ for $l = s, \ldots, 1$, and $ch_0>1$.
As shown in the next section, even for reasonably large $c^2$,
the coarsest-level problems are small. %; see Table~\ref{tbl:size_n0}. 

The parameter $\delta$ in~(\ref{eqn:Bl}) should be chosen to ensure the balance between computational
costs and the quality of the MG preconditioner. In particular, if $\delta$
is reasonably large then the choices of $B_l$ are dominated by the option
$B_l = L_l$, which is inexpensive but may not be suitable for larger shifts
on coarser levels. 
%in terms of the preconditioning accuracy. 
%
On the other extreme, if $\delta$ is close to zero then the common 
choice corresponds to $B_l = p_{m_l}(L_l - c^2 I_l)$, which provides a 
better preconditioning accuracy for larger shifts but may be too
computationally intense on finer levels.  
In our numerical experiments, % of the next section, 
we keep $\delta \in [1/3 ,3/4]$.

As we demonstrate in the next section, the degrees $m_l$ of the occurring polynomials $p_{m_l}$
should not be large, i.e., only a few matrix-vector multiplications with $L_l - c^2 I_l$ 
are required to obtain satisfactory approximations of absolute value operators. 
For properly chosen $\delta$, these additional 
multiplications need to be performed on grids that are 
significantly coarser than the finest grid, i.e.,
the involved matrices $L_l - c^2 I_l$ are orders of magnitude
smaller than the original fine grid operator.
%; 
%see Table~\ref{tbl:size_n0}.
% 
As confirmed by our numerical experiments, the overhead caused by the polynomial approximations % in~Algorithm~\ref{alg:avp-gmg-poly}
appears to be marginal and does not affect much the computational cost of
the overall preconditioning scheme.

\section{Numerical experiments}\label{sec:numeric}

This section presents a numerical study of the MG preconditioner
in Algorithm~\ref{alg:avp-gmg}. 
Our goal here is twofold.
On the one hand, the reported numerical experiments serve 
as a proof of concept of the AV preconditioning
described in Section~\ref{sec:absval_prec}.  
In particular, we show %on the example of the model problem
that the AV preconditioners can be constructed at
essentially the same cost as the standard preconditioning methods
(MG in our case). 
On the other hand, we demonstrate that the MG AV 
preconditioner in Algorithm~\ref{alg:avp-gmg} combined with
the optimal PMINRES iteration, in fact, 
leads to an efficient and economical computational scheme,
further called MINRES-AV-MG, 
which outperforms several known competitive approaches 
for the model problem.

Let us briefly describe the alternative preconditioners 
used for our comparisons. Throughout, we use {\sc matlab} for our numerical examples. 

\paragraph{\emph{\textbf{The inverted Laplacian preconditioner}}} 
This strategy, introduced in~\cite{Bayliss.Goldstein.Turkel:83}, 
is a representative of an SPD preconditioning for model problem~(\ref{eqn:helmholtz_fd}),
where
%and hence PMINRES is used as a method of choice.
the preconditioner is applied through solving systems $L w = r$, i.e., $T = L^{-1}$.
As has been previously discussed, for relatively small shifts $c^2$, the Laplacian $L$
constitutes a good SPD approximation of $|L - c^2 I|$. In this sense,
the choice $T = L^{-1}$ perfectly fits, as a special case,
into the general concept of the AV preconditioning presented
in Section~\ref{sec:absval_prec}.
We refer to PMINRES with $T = L^{-1}$ as MINRES-Laplace.

%The approach based on the inverted Laplacian was further modified 
%in~\cite{Laird:Giles:02} by introducing a shift into
%the preconditioner. It was suggested  
%to apply preconditioning by solving systems $(L + c^2 I) w = r$, i.e.,
%use $T = (L+c^2 I)^{-1}$. Clearly, in this case $T$ remains SPD, and
%PMINRES can be chosen as an appropriate iterative method. We call 
%the resulting scheme MINRES-Laplace-Shift.  

Usually, one wants to solve the system $Lw = r$ 
only approximately, i.e., use $T \approx L^{-1}$. 
This can be efficiently done, e.g., by applying the 
V-cycle of a standard MG method~\cite{Briggs.Henson.McCormick:00, Trottenberg.Oosterlee.Schuller:01}.
In our tests, however, we perform the exact solves using the {\sc matlab}'s ``backslash'', so that 
the reported results reflect the best possible convergence with the inverted Laplacian
type preconditioning.

%Throughout, we use {\sc matlab} for our numerical examples. The solution of systems $Lw = r$ for
%MINRES-Laplace is performed using the ``backslash'' operator. 
%We utilize the standard codes for MINRES and GMRES that are available in the {\sc matlab}'s collection 
%of iterative solvers. 
%
%Note that such a V-cycle can be obtained from
%Algorithm~\ref{alg:avp-gmg} if $B_l = L_l$ on all levels and coarse-grid 
%step~(\ref{eqn:mg-cgc-1}) is replaced by the linear solve with $L_0$. 
%%\begin{equation}\label{eqn:mg-cgc-1-laplace}
%%$w_0  =   \left( L_0 \right)^{-1} R_0 \left(r_1 - L_1 w_1^{pre}\right)$.
%%\end{equation}
%We refer to PMINRES preconditioned with this V-cycle as MINRES-Laplace-MG.

\paragraph{\emph{\textbf{The indefinite MG preconditioner}}} % \label{subsubsec:laplace_prec}
We consider a standard V-cycle for problem~(\ref{eqn:helmholtz_fd}). Formally, it can be 
obtained from Algorithm~\ref{alg:avp-gmg} by setting $B_l = L_l - c^2 I_l$ on all levels and
replacing the first equality in~(\ref{eqn:mg-cgc-1}) by the linear solve with $L_0 - c^2 I_0$. 
The resulting MG scheme is used as a preconditioner for restarted GMRES and for Bi-CGSTAB.
We refer to these methods as GMRES($k$)-MG and Bi-CGSTAB-MG, respectively; $k$ denotes the restart parameter.  
A thorough discussion of the indefinite MG preconditioning 
%combined with GMRES iterations 
for Helmholtz problems can be found, e.g., in~\cite{Elman.Ernst.OLeary:01}.

%%
%\begin{table}[ht]
%\begin{center}
%\renewcommand{\arraystretch}{1.2}
%\begin{tabular}{|c|c|c|c|c|c|c|}  \hline 
%               & $c^2 = 300$ & $c^2 = 400$  & $c^2 = 1000$ & $c^2 = 1200$ & $c^2 = 1500$ & $c^2 = 3000$\\ \hline \hline
%$\delta = 1/3$ & $961$       & $961$        & $3969$       & $3969$       & $3969$       &  $16129$       \\ \hline
%$\delta = 1/2$ & $961$       & $961$        & $961$        & $3969$       & $3969$       &  $3969$       \\ \hline
%%$\delta = 3/4$ & $225$       & $225$        & $961$        & $961$        & $961$        &  $3969$       \\ \hline
%$\delta = 1$   & $225$       & $225$        & $225$        & $961$        & $961$        &  $961$       \\ \hline
%\end{tabular}
%\vspace{.1cm}
%\end{center}
%\caption{The largest problem sizes satisfying $ch_l \geq \delta$
%for different values of the shift $c^2$, ``switching''
%parameters $\delta$, and the standard coarsening scheme $h_{l-1} = 2 h_l$.
%The last row ($\delta = 1$) corresponds to the sizes of the coarsest
%problems for different $c^2$.}
%%Standard coarsening is applied to model problem~(\ref{eqn:helmholtz_bvp}) 
%%discretized on the fine grid with $h = 2^{-j}$, where $j$ is a given number.}  
%\label{tbl:size}
%\end{table}

%
\begin{table}[ht]
\caption{The largest problem sizes satisfying $ch_l \geq \delta$
for different values of the shift $c^2$, ``switching''
parameters $\delta$, and the standard coarsening scheme $h_{l-1} = 2 h_l$.
The last row ($\delta = 1$) corresponds to the sizes of the coarsest
problems for different $c^2$.}
\label{tbl:size}
\begin{center}
\renewcommand{\arraystretch}{1.2}
\begin{tabular}{|c|c|c|c|c|c|}  \hline 
               & $c^2 = 300$ & $c^2 = 400$  &  $c^2 = 1500$ & $c^2 = 3000$ & $c^2=4000$\\ \hline \hline
$\delta = 1/3$ & $961$       & $961$        & $3969$       &  $16129$      & $16129$   \\ \hline
$\delta = 1/2$ & $961$       & $961$        & $3969$       &  $3969$       & $3969$     \\ \hline
$\delta = 3/4$ & $225$       & $225$        & $961$        &  $3969$       & $3969$    \\ \hline
$\delta = 1$   & $225$       & $225$        & $961$        &  $961$        & $961$    \\ \hline
\end{tabular}
\vspace{.1cm}
\end{center}
\end{table}

%Unless otherwise explicitly stated, we consider 2D model problem~(\ref{eqn:helmholtz_fd}) 
In our tests, we consider 2D model problem~(\ref{eqn:helmholtz_fd}) 
corresponding to~(\ref{eqn:helmholtz_bvp}) discretized on the grid of size $h = 2^{-8}$ 
(the fine problem size $n = 65025$).
The exact solution $x^*$ and the initial guess $x_0$ are randomly chosen.
The right-hand side $b = (L - c^2 I) x^*$, which allows evaluating the actual errors
along the steps of an iterative method.  
%
%The right-hand side $b$ and the initial guess $x_0$ are randomly chosen.
%
All the occurring MG preconditioners are built upon the standard coarsening scheme (i.e., $h_{l-1} = 2h_l$), 
restriction is based on the full weighting, and 
prolongation on piecewise multilinear 
interpolation~\cite{Briggs.Henson.McCormick:00, Trottenberg.Oosterlee.Schuller:01}.
%Throughout, we use {\sc matlab} for our numerical examples. 
%The solution of systems $Lw = r$ for
%MINRES-Laplace is performed using the ``backslash'' operator. 
%and utilize the standard codes for MINRES and GMRES that are available in the {\sc matlab}'s collection 
%of iterative solvers. 

Let us recall that Algorithm~\ref{alg:avp-gmg} requires setting a parameter $\delta$
to switch between $B_l  = L$ and $p_{m_l}(L - c^2 I)$ on different levels; see~(\ref{eqn:Bl}).
Assuming standard coarsening, Table~\ref{tbl:size} presents the largest problem sizes
corresponding to the condition $ch_l \geq \delta$ for a few values of $\delta$ and $c^2$. 
In other words, given $\delta$ and $c^2$, each cell of Table~\ref{tbl:size} contains the
largest problem size for which the polynomial approximation of $|L_l - c^2 I_l|$ is constructed. 
Unless otherwise explicitly stated, we set $\delta = 1/3$. 
Note that according to the discussion in subsection~\ref{subsec:mg-avp} (condition~(\ref{eqn:ch_coarse})), 
the row of Table~\ref{tbl:size} corresponding to $\delta = 1$ delivers the
sizes $n_0$ of the coarsest problems for different shift values. 

Table~\ref{tbl:size} shows that the coarsest problems remain relatively small even
for large shifts. The polynomial approximations are constructed
for coarser problems of significantly reduced dimensions, which in practical applications
are negligibly small compared to the original problem size. % on the finest grid.  

As a smoother on all levels of Algorithm~\ref{alg:avp-gmg} 
we use Richardson's iteration, i.e., $M_l^{-1} \equiv \tau_l I_l$.
On the finer levels, where $B_l = L_l$, we 
%configure Algorithm~\ref{alg:avp-gmg} to 
%apply one (pre- and post-) smoothing step 
%of 
%$\omega$-damped Jacobi 
%Richardson's iteration, i.e., $M_l^{-1} \equiv \tau_l I_l$
%and $\nu_l = 1$, as a (pre- and post-) smoother 
%with 
choose $\tau_l = h^2_l/5$ and $\nu_l = 1$. 
%damping parameter $\omega = 4/5$~\cite{Briggs.Henson.McCormick:00, Trottenberg.Oosterlee.Schuller:01}.
%
%The same smoothing scheme is used on all levels of the inverted Laplacian type 
%preconditioner in MINRES-Laplace-MG. 
%
%
On the coarser levels, where $B_l = p_{m_l}(L_l - c^2 I_l)$, 
%Algorithm~\ref{alg:avp-gmg} performs
we set 
%perform smoothing  
%based on Richardson's iteration, i.e., $M_l^{-1} \equiv \tau_l I_l$.
%We choose 
$\tau_l = h_l^2/(5 - c^2 h_l^2)$ and $\nu_l = 5$. 

Similar to the 1D case considered in subsections~\ref{subsec:BLaplacian} and~\ref{subsec:Bpoly},
%Note that 
both choices of $\tau_l$ correspond to the
optimal smoothing of oscillatory eigenmodes with respect to the 2D operators 
$L_l$ and $|L_l -c^2 I_l|$, respectively~\cite{Briggs.Henson.McCormick:00, Trottenberg.Oosterlee.Schuller:01}.
%so that for the ideal case 
%$p_{m_l} (L_l - c^ 2 I_l) = |L_l - c^2 I_l|$ the 
%optimal smoothing of oscillatory eigenmodes
%is attained; 
%similar to the 1D case considered in subsection~\ref{subsec:Bpoly}.
%
Since $p_{m_l} (L_l -c^2 I_l)$ only approximates $|L_l - c^2 I_l|$ in practice, the choice of $\tau_l$
on the coarser grids is likely to be not optimal. Therefore, for $ch_l \geq \delta$, we have increased the number of smoothing
steps to $5$. 
%on coarser levels, i.e., $\nu_l = 5$ for $ch_l \geq \delta$.  
In all tests, the degrees $m_l$ of polynomials are set to $10$.
The intervals containing $\Lambda(L_l - c^2 I_l)$, required for evaluating
%the Chebyshev polynomials, 
$p_{m_l}$, are $[-c^2,8h_l^{-2} - c^2]$. 
The inverted coarse grid absolute value $|L_0 - c^2 I_0|^{-1}$ is constructed by the full eigendecomposition.

\begin{figure}[ht]
\begin{center}%
\begin{tabular}
[c]{cc}%
    \includegraphics[width=6.5cm]{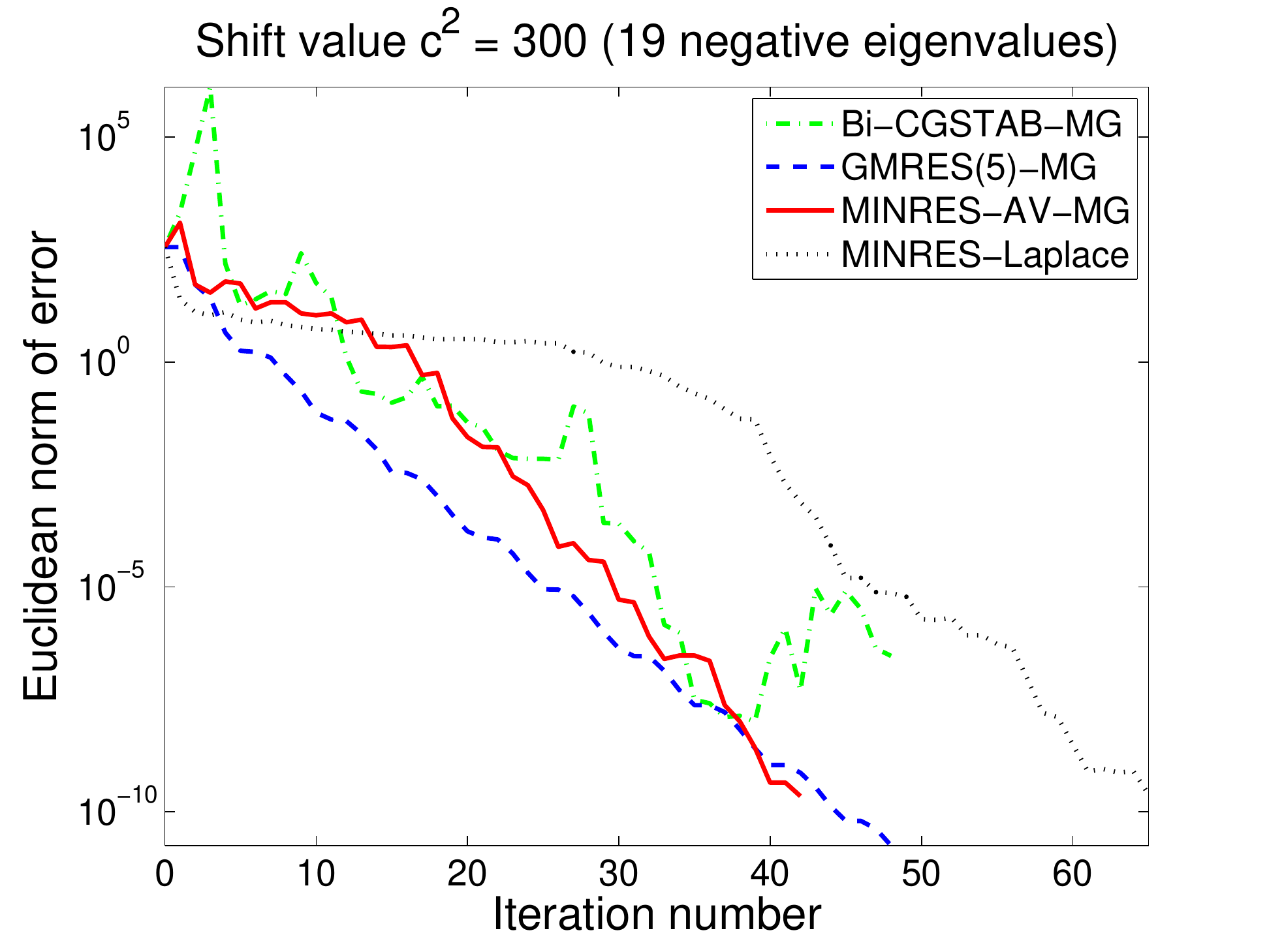}
    \includegraphics[width=6.5cm]{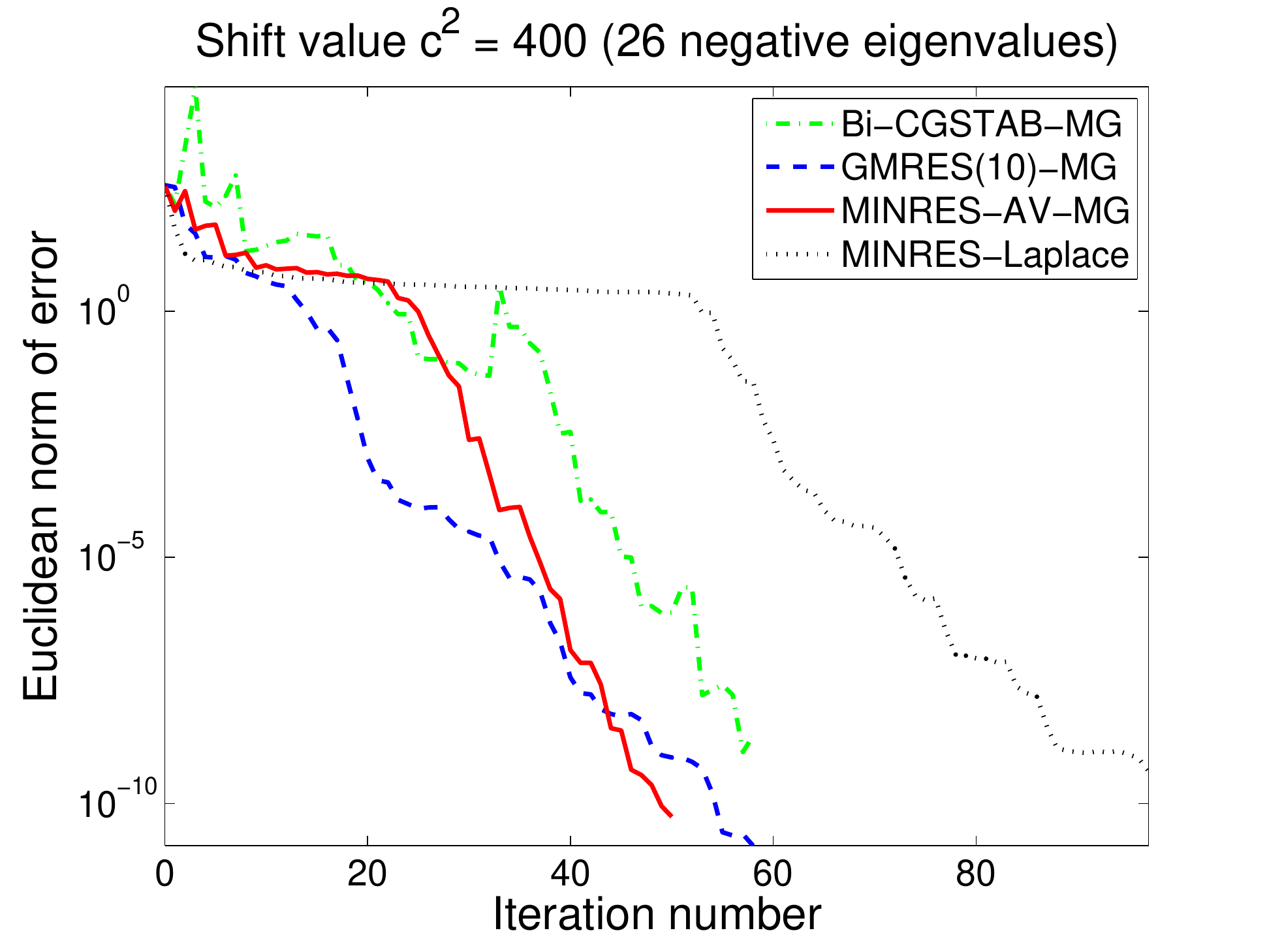} \\
    \includegraphics[width=6.5cm]{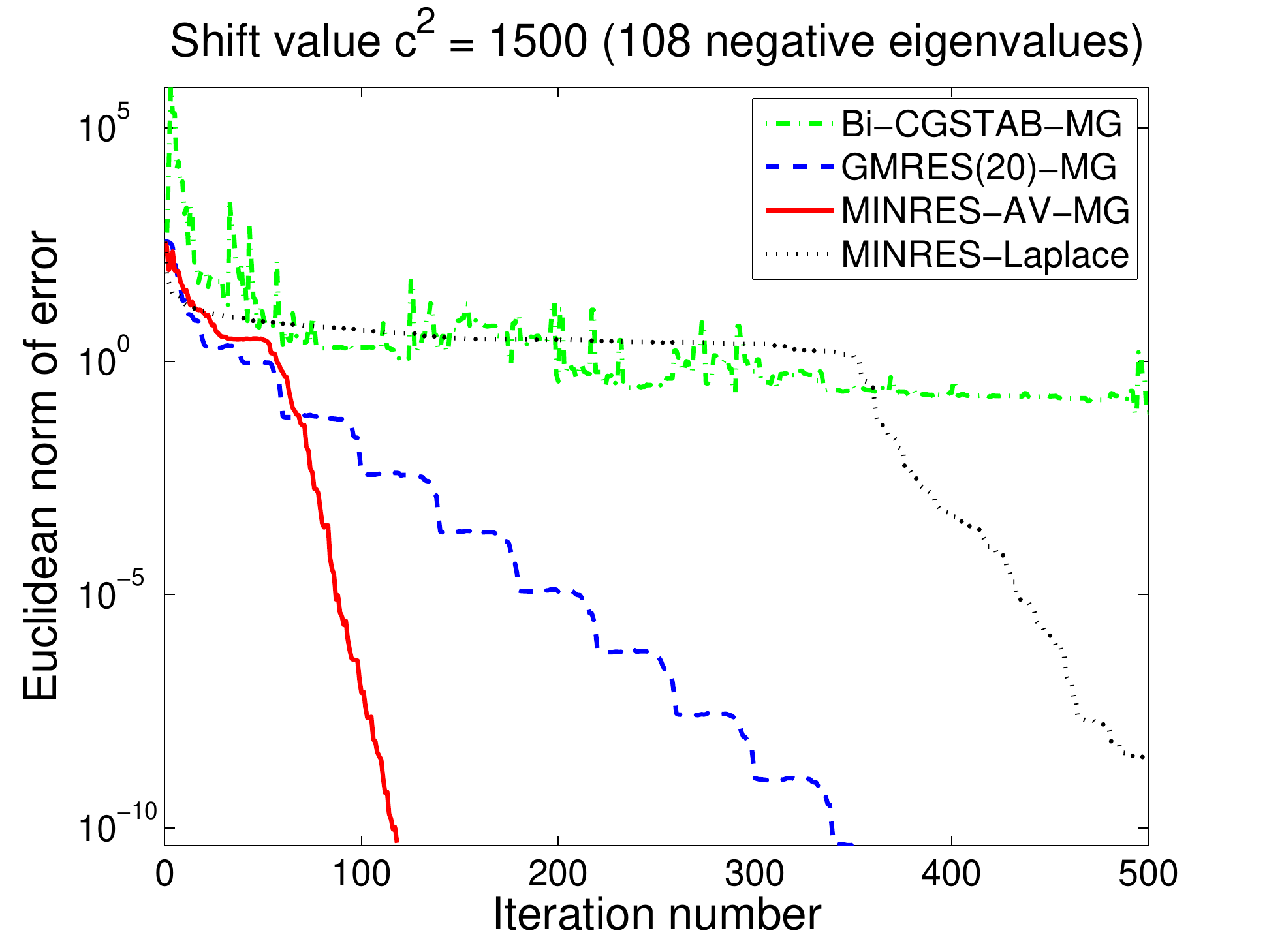}
    \includegraphics[width=6.5cm]{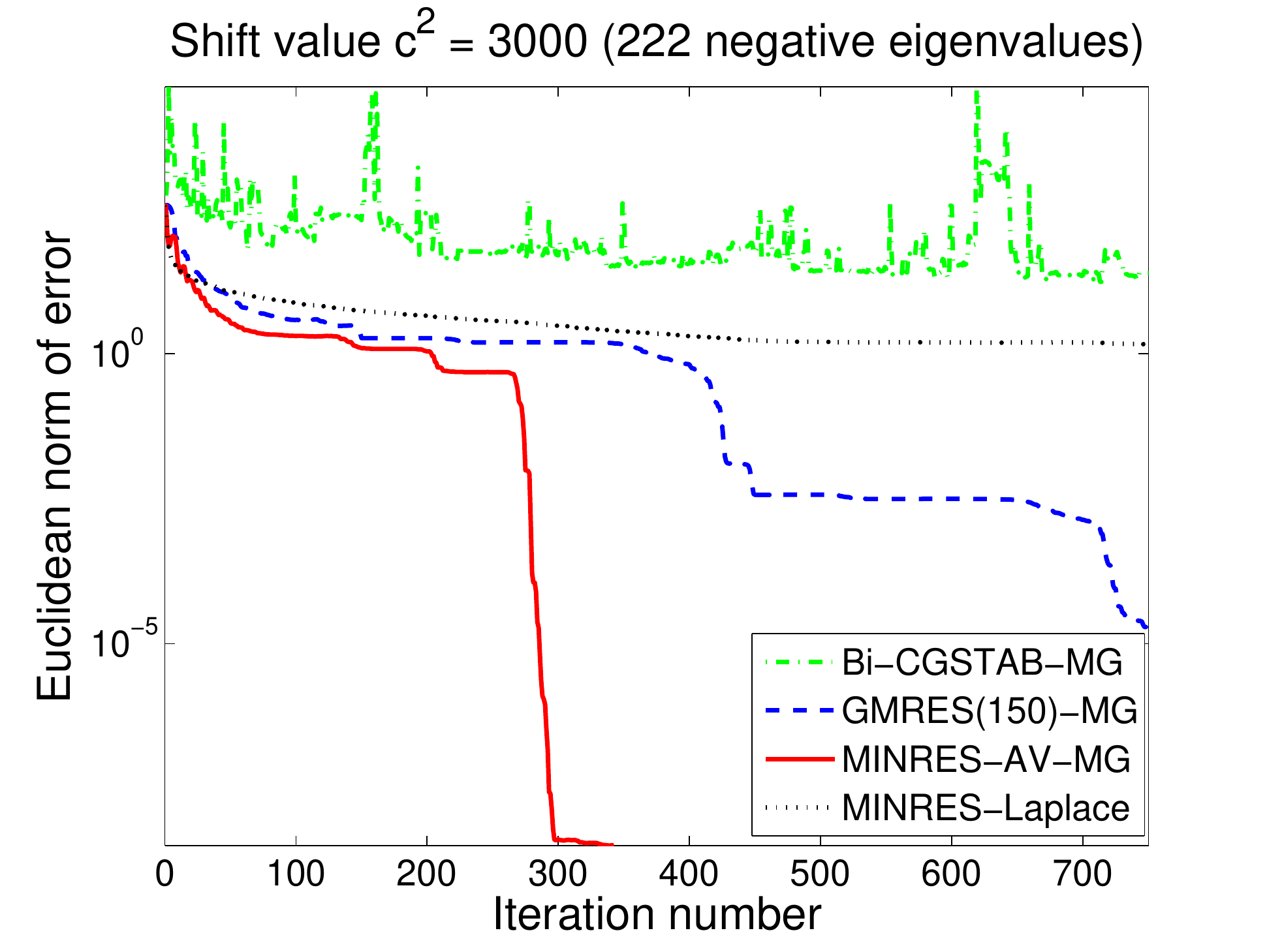}
\end{tabular}
\end{center}
\caption{Comparison of several preconditioned schemes;  
$n = 65025$.}%
\label{fig:prec}%
\end{figure}

In Figure~\ref{fig:prec}, we compare MINRES-AV-MG with the above 
introduced alternative preconditioned schemes for the model problem. 
Each plot corresponds to a different shift value. 
The restart parameter $k$ varies for all runs of GMRES($k$)-MG, increasing 
(left to right and top to bottom) as $c^2$ grows from $300$ to $3000$.
In our tests, the size $n_0$ of the coarsest problem in Algorithm~\ref{alg:avp-gmg} is $225$ (Figure~\ref{fig:prec}, top) 
and $961$ (Figure~\ref{fig:prec}, bottom);
see Table~\ref{tbl:size} with $\delta = 1$. The same 
$n_0$ is used for the MG preconditioner in the corresponding runs of GMRES($k$)-MG and Bi-CGSTAB-MG.

%The values $k$ in Figure~\ref{fig:prec} have been chosen to compromise between the storage expenses of
%GMRES($k$)-MG and its convergence behavior.    
%For a fair comparison, we set $k$ to be sufficiently small, so that the storage
%required for GMRES($k$)-MG is as close as possible to that of MINRES-AV-MG,
%while the convergence of the method is not lost. 
%%
%Since Bi-CGSTAB-MG is based on a short-term recurrence, its storage expenses
%are similar to MINRES-AV-MG.  

Figure~\ref{fig:prec} shows 
that MINRES-AV-MG noticeably outperforms
PMINRES with the inverted Laplacian preconditioner.
For smaller shifts ($c^2 = 300, \: 400$), MINRES-AV-MG is comparable, in terms
of the iteration count, to GMRES($k$)-MG and Bi-CGSTAB-MG; $k = 5, \: 10$. 
For larger shifts ($c^2 = 1500, \: 3000$), however, 
MINRES-AV-MG provides a superior convergence behavior. 
In particular, the scheme exhibits faster convergence than
GMRES($k$)-MG under less demanding storage requirements,
while Bi-CGSTAB-MG fails to converge~($k = 50, \: 150$).

If the polynomial approximations in Algorithm~\ref{alg:avp-gmg} appear only on sufficiently coarse 
grids and the size $n_0$ of the coarsest problem is relatively small, 
%i.e., the coarsest-level eigendecomposition is not much more expensive than a linear solve, 
then the additional costs introduced by the coarser grid computations 
of the AV MG preconditioner are     
negligible relative to the cost of operations on the finer grids, which are
the same as in the standard V-cycle for the indefinite problem. 
This means that the complexity of Algorithm~\ref{alg:avp-gmg} 
is similar to that of the MG preconditioner in GMRES($k$)-MG and Bi-CGSTAB-MG.
%
%Indeed, the runs of MINRES-AVP-MG on both plots in Figure~\ref{fig:spd_prec} 
%have been around $1.6$ times faster in wallclock time than those of MINRES-Laplace-MG, which is compatible 
%with the reduction in the iteration count. 
%
%
%We have already discussed that Algorithm~\ref{alg:avp-gmg} has essentially 
%the same cost as the V-cycle for the Poisson equation used as a preconditioner
%in MINRES-Laplace-MG, provided that the grids corresponding to $ch_l \geq \delta$ are sufficiently
%coarse. 
%
%The same considerations can be applied to show that the complexity 
%of Algorithm~\ref{alg:avp-gmg} is similar to 
%the one of the indefinite MG preconditioner in GMRES($k$)-MG.
 
To be more precise, in the tests reported in Figure~\ref{fig:prec} (bottom),
a single application of Algorithm~\ref{alg:avp-gmg} has required $15$--$20\%$
more time than the indefinite MG preconditioner, even though the polynomial approximations
in Algorithm~\ref{alg:avp-gmg} have been constructed for problem sizes as large as $16129$ if $c^2 = 3000$,
which is of the same order as the original problem size $n$.
For larger problem sizes, the time difference becomes negligible.  
For example if $h = 2^{-9}$ ($n = 261121$), Algorithm~\ref{alg:avp-gmg} results in only $5\%$
time increase, 
%compared to the indefinite MG preconditioner, 
and the relative time difference becomes 
indistinguishable for smaller $h$. 
The application of all the MG preconditioners in Figure~\ref{fig:prec} (top) 
required essentially the same time.  

The above discussion suggests that our numerical comparison, based on the number of iterations, is representative.
Additionally, in Table~\ref{tbl:time} we provide the time comparisons for the MG preconditioned
schemes. In particular, we measure the actual time required
by the runs of MINRES-AV-MG ($t_{AV}$), GMRES($k$)-MG ($t_G$), and Bi-CGSTAB-MG ($t_B$)
in Figure~\ref{fig:prec}, and report the speed-ups. 

\begin{table}[ht]
\caption{
Time comparison of the MG preconditioned schemes in Figure~\ref{fig:prec}.}
\label{tbl:time}
\begin{center}
\renewcommand{\arraystretch}{1.2}
\begin{tabular}{|c|c|c|c|c|}  \hline 
               & $c^2 = 300$ & $c^2 = 400$  &  $c^2 = 1500$ & $c^2 = 3000$\\ \hline \hline
$t_B/t_{AV}$ & $1.1$       & $1.1$        & $-$       &  $-$       \\ \hline
$t_G/t_{AV}$ & $1.4$       & $1.3$        & $2.6$       &  $1.9$       \\ \hline
\end{tabular}
\vspace{.1cm}
\end{center}
\end{table}

We have observed that the performance of
GMRES($k$)-MG can be improved by increasing the restart parameter. 
In Figure~\ref{fig:prec}, however, the values $k$ have been chosen to balance between storage 
%expenses of GMRES($k$)-MG and its 
and convergence behavior.    
%The values $k$ in Figure~\ref{fig:prec} 
%have been chosen to compromise between the storage expenses of
%GMRES($k$)-MG and its convergence behavior.    
%For a fair comparison, 
In particular, we set $k$ to be sufficiently small, so that the storage
required for GMRES($k$)-MG is as close as possible to that of MINRES-AV-MG,
while the convergence of the method is not lost. 
Since Bi-CGSTAB-MG is based on a short-term recurrence, its storage expenses
are similar to that MINRES-AV-MG.

The unsatisfactory performance of GMRES($k$)-MG in Figure~\ref{fig:prec} can, in part, 
be attributed to the observation that smoothing based on Richardson's (or, more generally, Jacobi) 
iteration becomes increasingly unstable as grids coarsen.
In particular, as shown in~\cite{Elman.Ernst.OLeary:01}, on the intermediate levels with $ch_l \geq 1/2$
this smoothing scheme strongly amplifies the smooth error eigenmodes.
A straightforward remedy is to invoke the coarse grid solve 
%on a sufficiently fine grid, i.e., 
on the largest grid that fails to satisfy $ch_l < 1/2$.

%We have observed that the performance of
%GMRES($k$)-MG can be improved by increasing either the restart parameter
%or the size of the coarsest problem. 
%%
%However, for $n$ sufficiently large, the former may not be feasible, while the latter
%increases the cost of the preconditioning.

% -------- begin NEW TABLES ------------------

%%
%\begin{table}[ht]
%\caption{
%Number of iterations.}
%%Time comparison of the MG preconditioned schemes in Figure~\ref{fig:prec}.}
%\label{tbl:fgmres_iter}
%\begin{center}
%\renewcommand{\arraystretch}{1.2}
%\begin{tabular}{|c|c|c|c|c|c|c|}  \hline 
%               & MINR-AV      & FG(5)      & FG(10)     & FG(20)      & FG(25)      & FG(35)\\ \hline \hline
%$c^2 = 1500$   & $80$     & $27(180)$  & $12(427)$  &  $12(296)$  & $12(296)$   & $12(296)$  \\ \hline
%$c^2 = 3000$   & $257$    & $-$     & $-$        &  $-$           & $147(?)$       & $72$      \\ \hline
%$c^2 = 4000$   & $290 (383)$  & $-$     & $-$        &  $-$      & $834$       & $239$      \\ \hline
%\end{tabular}
%\vspace{.1cm}
%\end{center}
%\end{table}

\begin{table}[ht]
\caption{
Number of iterations of MINRES-AV-MG (MINR) and GMRES($k$)-MG (GMR($k$)) required to reduce 
the initial error by $10^{-8}$; $n = 65025$. The preconditioner in GMRES($k$)-MG uses Richardson's smoothing on levels
$ch_l<1/2$ and invokes the coarse grid solve on the level that follows. Numbers in the parentheses correspond to
the right preconditioned GMRES($k$)-MG. Dash denotes that the method failed to converge within $1000$ steps.}
%Time comparison of the MG preconditioned schemes in Figure~\ref{fig:prec}.}
\label{tbl:early_coarse}
\begin{center}
\renewcommand{\arraystretch}{1.2}
\begin{tabular}{|c|c|c|c|c|c|c|}  \hline 
               & MINR   & GMR(5)          & GMR(10)     & GMR(20)       & GMR(25)     & GMR(35) \\ \hline \hline
$c^2 = 1500$   & $89$      & $29 (44)$        & $20 (22)$    & $16 (18)$      & $16 (18)$    &  $16 (18)$  \\ \hline
$c^2 = 3000$   & $282$     & $-  (-)$         & $- (-)$      &  $- (-)$       & $223 (269)$  &  $69 (69)$      \\ \hline
$c^2 = 4000$   & $310$       & $- (-)$          & $- (-)$      &  $- (-)$       & $- (-)$      &  $395 (471)$      \\ \hline
\end{tabular}
\vspace{.1cm}
\end{center}
\end{table}

% -------- end NEW TABLES ------------------

In Table~\ref{tbl:early_coarse}, we compare MINRES-AV-MG and GMRES($k$)-MG with different
values of the restart parameter. 
We report the iteration counts required to reduce the initial error by $10^{-8}$ for systems with 
$c^2 = 1500, 3000$, and $4000$. 
The indefinite MG preconditioner in GMRES($k$)-MG is configured to run 
Richardson's smoothing on grids $ch_l < 1/2$ and perform the coarse grid solve on the level 
that follows. We test the indefinite MG preconditioner in the left (used so far) 
and right preconditioned versions of GMRES($k$).
  
The above described setting of the right preconditioned GMRES($k$)-MG 
represents a special case of the Helmholtz solver introduced in~\cite{Elman.Ernst.OLeary:01}.
In this paper, 
instead of the ``early'' coarse grid solve, the MG preconditioning scheme performs further coarsening
on levels $ch_l \geq 1/2$ with GMRES used as a smoother. Since the resulting preconditioner
is nonlinear, (the full) FGMRES~\cite{Saad:93} is used for outer iterations.

It is clear that the results in Table~\ref{tbl:early_coarse} 
provide an insight into the best possible outer iteration counts that can be expected from the \textit{restarted} 
version of the method in~\cite{Elman.Ernst.OLeary:01}. In fact, the same observation is used by the authors
of~\cite{Elman.Ernst.OLeary:01}
for hand-tuning the smoothing schedule of their preconditioner to study its best-case performance.
%
%replace the coarse grid solve by further coarsening with GMRES used 
%%a Krylov subspace method
%as a smoother. In this case, the coarsest grid
%%, where the direct solve is performed, 
%can contain only a few nodes. On levels $ch_l < 1/2$, Richardson's (Jacobi) iteration is left as a
%
Table~\ref{tbl:early_coarse} then demonstrates that, regardless of the smoothing schedule and 
the smoother's stopping criterion,
%
%, even with the ``early''
%coarse grid solve, GMRES($k$)-MG 
the (F)GMRES based solvers, such as, e.g., the one in~\cite{Elman.Ernst.OLeary:01}, 
require increasing storage
%exhibit an increasing storage requirement 
to maintain the convergence as $c^2$ grows,
whereas the robustness of
MINRES-AV-MG is not lost under the minimalist memory~costs.

Thus, if the shift is large and the amount of storage is limited, so that $k$ is forced to be sufficiently small, 
the (F)GMRES($k$) outer iterations may fail to converge within a reasonable number of steps, even if the coarse grid solve 
in the MG preconditioner is performed ``early'' . 
We note, however, that if storage is available or the shifts
are not too large, the (F)GMRES based methods  
%schemes, similar to the one in~\cite{Elman.Ernst.OLeary:01} 
may represent a valid~option. 
%see Table~\ref{tbl:early_coarse} for $c^2 = 1500$.    

%It is clear that the use of the indefinite MG preconditioner 
%%with the ``early'' coarse grid solve, 
%%on the level following the grids with $ch_l < 1/2$, 
%as in Table~\ref{tbl:early_coarse} is generally not suitable for practical computations because
%the size of the coarsest problem can be 
%%too large.
%prohibitively large.
%%
%This issue has been addressed in~\cite{Elman.Ernst.OLeary:01}, where the authors 
%%suggest to 
%replace the coarse grid solve by further coarsening with GMRES used 
%%a Krylov subspace method
%as a smoother. In this case, the coarsest grid
%%, where the direct solve is performed, 
%can contain only a few nodes. On levels $ch_l < 1/2$, Richardson's (Jacobi) iteration is left as a
%smoother. Since the resulting preconditioner is nonlinear, the outer GMRES iterations are replaced
%with~FGMRES~\cite{Saad:93}.    
%  
%The performance of the method in~\cite{Elman.Ernst.OLeary:01} significantly depends on the stopping
%criterion for the GMRES smoothing.
%However, 
%the results for the right preconditioned 
%GMRES($k$)-MG 
%%with the ``early'' coarse grid solve
%in Table~\ref{tbl:early_coarse}
%provide an indication of the best possible FGMRES iteration counts that can be expected from the restarted version of the scheme in~\cite{Elman.Ernst.OLeary:01},
%regardless of the stopping rule. 
%Therefore, the above conclusions for GMRES($k$)-MG, motivated by Table~\ref{tbl:early_coarse}, also apply to the method in~\cite{Elman.Ernst.OLeary:01}
%with restarts.  

%Although not reported in this paper due to space limitations, 
We have also tested the \textit{Bunch-Parlett factorization}~\cite{Gill.Murray.Ponceleon.Saunders:92}  
as a coarse grid solve in the MG framework. 
In particular, as a preconditioner for MINRES, 
we have used Algorithm~\ref{alg:avp-gmg} with $B_l = L_l$ on all levels and the coarsest-grid 
absolute value in~(\ref{eqn:mg-cgc-1}) replaced by the application of the ``perfect''
Bunch-Parlett factorization based preconditioner~\cite{Gill.Murray.Ponceleon.Saunders:92}. 
We have obtained results that are inferior to
%MINRES-AV-MG, MINRES-Laplace, and MINRES-Laplace-Shift
the schemes considered in this paper
for shifts not too small, e.g., for $c^2 > 200$ if $h = 2^{-8}$.  
The unsatisfactory behavior may be related to the fact that the inverted Laplacian
$T = L^{-1}$ and the ideal absolute value $T = \left| L - c^2 I \right|^{-1}$ preconditioners
share the same eigenvectors with $A = L - c^2 I$, while the %``perfect''
preconditioner from~\cite{Gill.Murray.Ponceleon.Saunders:92} does not.

The standard MG preconditioners, as in GMRES($k$)-MG and Bi-CGSTAB-MG, 
are known to have optimal costs, linearly proportional to $n$. 
As discussed above, the same is true for the AV 
preconditioner in Algorithm~\ref{alg:avp-gmg}.  
Therefore, if, in addition, the number of iterations in the iterative solver preconditioned with
Algorithm~\ref{alg:avp-gmg} does not depend on the problem size, the overall 
scheme is optimal.

%\begin{table}[ht]
%\begin{center}
%\renewcommand{\arraystretch}{1.2}
%\begin{tabular}{|c|c|c|c|c|c|c|}  \hline 
%            & $h = 2^{-5}$ & $h = 2^{-6}$ & $h = 2^{-7}$ & $h = 2^{-8}$ & $h = 2^{-9}$ & $h = 2^{-10}$ \\ \hline \hline
%$c^2 = 100$ & $14$ & $14$ & $15$       & $14$         & $14$         & $14$          \\ \hline
%$c^2 = 200$ & $21$ & $21$ & $21$       & $21$         & $21$         & $21$          \\ \hline
%$c^2 = 300$ & $31$ & $32$ & $31$       & $32$         & $32$         & $30$          \\ \hline
%$c^2 = 400$ & $40$ & $40$ & $40$       & $39$         & $40$         & $40$          \\ \hline
%\end{tabular}
%\vspace{.1cm}
%\end{center}
%\caption{Mesh-independent convergence of PMINRES with the MG absolute value preconditioner.} % for the model problem.}
%\label{tbl:h_indep}
%\end{table}

\begin{table}[ht]
\caption{Mesh-independent convergence of PMINRES with the MG AV preconditioner.
The numbers in parentheses correspond $\delta = 3/4$. The default value of $\delta$ is $1/3$.} % for the model problem.}
\label{tbl:h_indep}
\begin{center}
\renewcommand{\arraystretch}{1.2}
\begin{tabular}{|c|c|c|c|c|c|c|c|}  \hline 
             & $h = 2^{-6}$  & $h = 2^{-7}$ & $h = 2^{-8}$ & $h = 2^{-9}$ &  $h =2^{-10}$ & $h = 2^{-11}$              \\ \hline \hline
$c^2 = 300$  & $31 (31)$     & $31 (31)$    & $30 (32)$    & $30 (32)$    &  $30 (32)$    &  $30 (30)$                   \\ \hline
$c^2 = 400$  & $37 (40)$     & $38 (40)$    & $37 (40)$    & $37 (40)$    &  $37 (40)$    &  $37 (39)$                \\ \hline
$c^2 = 1500$ & $67 (97)$     & $97 (119)$   & $89 (109)$   & $88(108)$    &  $89 (106)$   &  $90 (107)$              \\ \hline
$c^2 = 3000$ & $228 (229)$   & $222 (284)$  & $279(332)$   & $256(298)$   &  $257 (296)$  &  $256 (298)$             \\ \hline
\end{tabular}
\vspace{.1cm}
\end{center}
\end{table}

We verify this optimality in Table~\ref{tbl:h_indep}, which shows the mesh-independence of the convergence of
PMINRES with the MG AV preconditioner. 
% given by Algorithm~\ref{alg:avp-gmg}. 
The rows of the table correspond to the
shift values $c^2$, while the columns match the mesh size $h$. % underlying discrete problem~(\ref{eqn:helmholtz_fd}). 
The cell in the intersection
contains the numbers of steps performed to achieve the decrease by
the factor $10^{-8}$ in the error $2$-norm with the choices of the 
``switching'' parameter $\delta = 1/3$ and $\delta = 3/4$. 
%The number of steps
%corresponding to the latter value of $\delta$ is given in parentheses. 

As previously, the size of the coarsest 
grid has been set according to Table~\ref{tbl:size} with $\delta = 1$.
We conclude that the convergence does not slowdown with the decrease of $h$;
thus, PMINRES preconditioned by Algorithm~\ref{alg:avp-gmg} is optimal. 
Note that for larger shifts, $c^2 = 1500$ and $c^2 = 3000$, mesh-independent convergence
occurs for $h$ sufficiently small, when the ``switching'' pattern is stabilized, i.e.,
$B_l = L_l$ on a few finer grids and $B_l = p_{m_l}(L_l - c^2 I_l)$ on the coarser
grids that follow.

Table~\ref{tbl:h_indep} shows that as $c^2$ grows, the increase in the iteration count is mild and
essentially linear.
As expected, the smaller value of $\delta$, which leads to the construction of the
polynomial approximations earlier on finer levels, results in a higher accuracy of the AV preconditioner.

\begin{figure}[ht]
\begin{center}%
%\begin{tabular}[c]{cc}%
    \includegraphics[width = 13cm]{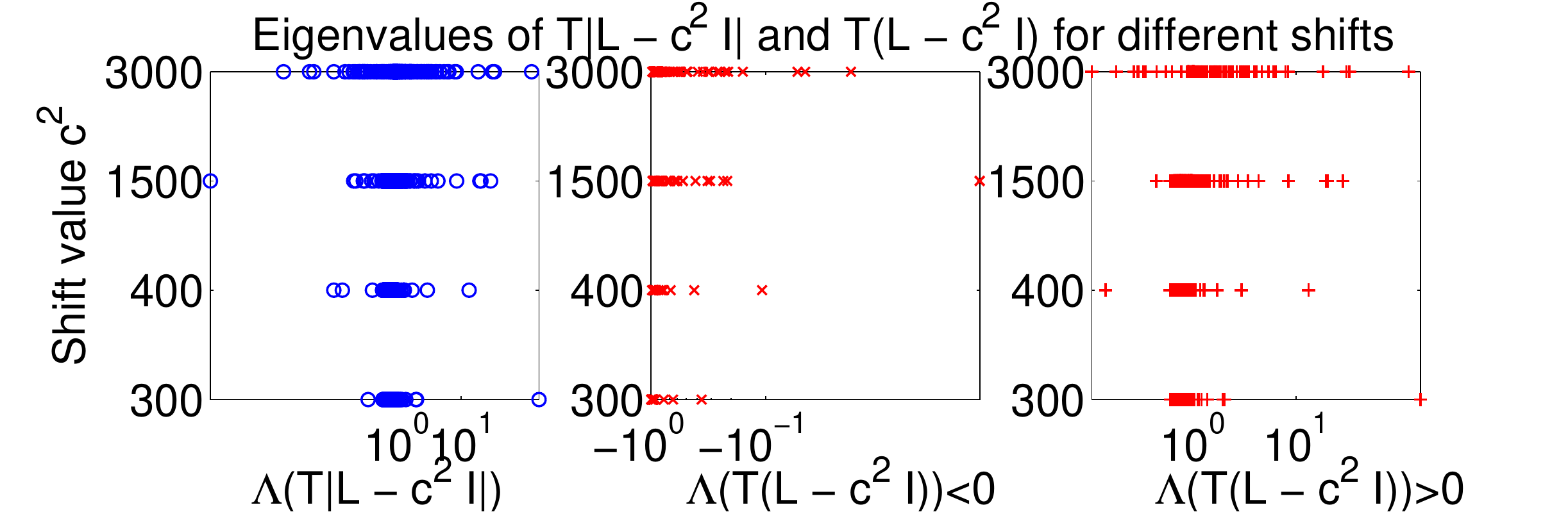}
%\end{tabular}
\end{center}
\caption{Spectrum of $T\left| L - c^2 I \right|$ (left), 
negative eigenvalues of $T\left( L - c^2 I \right)$ (center), and
positive eigenvalues of $T\left( L - c^2 I \right)$ (right); $n = 16129$.}%
\label{fig:spectr}%
\end{figure}

Finally, in Figure~\ref{fig:spectr} we plot the eigenvalues
of $T\left| L - c^2 I \right|$ and $T(L - c^2 I)$ for different shift values using the logarithmic scale; $n = 16129$.
As suggested by Corollary~\ref{cor:avp_cluster}, clusters of eigenvalues of $T\left|L - c^2 I\right|$
are preserved in the spectrum of the preconditioned matrix $T(L - c^2 I)$. 
%First, we observe that 
Almost all eigenvalues of $T(L - c^2 I)$ are clustered around $-1$ and $1$, with only a few 
%The few eigenvalues that 
falling outside of the clusters. 
%may suddenly increase the condition number 
%$\kappa(T\left|A\right|)$ with varying $h$, as seen in Table~\ref{tbl:delta}. 
%These outliers, however, do not noticeably affect PMINRES convergence, 
%as Figure~\ref{fig:spd_prec} and Table~\ref{tbl:h_indep} show.  
%
%Second, both 
We note that the clustering and the condition number $\kappa(T\left|A\right|)$ deteriorate
as $c^2$ increases from $300$ to $3000$, which is compatible with the results in Table~\ref{tbl:h_indep}. 
%This explains the slowdown of PMINRES in Figure~\ref{fig:spd_prec} and Table~\ref{tbl:h_indep} 
%with the increase of  $c^2$.  

The spectra computed in Figure~\ref{fig:spectr} allow validating numerically the 
tightness of bounds~(\ref{eqn:spectr_bounds}) in Theorem~\ref{thm:avp_spect_bounds}
for the MG AV preconditioner. In Table~\ref{tbl:bounds}, we report the number of eigenvalues $\lambda_j$
of $T(L - c^2 I)$ that satisfy either the upper or the lower bound up to machine precision. 
The table shows that the bound is numerically sharp.
\begin{table}[ht]
\caption{
Number of eigenvalues $\lambda_j$ that equal the upper/lower bound in~(\ref{eqn:spectr_bounds}) up to machine precision.}
\label{tbl:bounds}
\begin{center}
\renewcommand{\arraystretch}{1.2}
\begin{tabular}{|c|c|c|c|c|}  \hline 
              &  $c^2 = 300$ & $c^2 = 400$  &  $c^2 = 1500$ & $c^2 = 3000$\\ \hline \hline
  \mbox{Upper}            &  $0$       & $15$        & $1$       &  $115$       \\ \hline
  \mbox{Lower}            &  $0$       & $0$        & $10$       &  $0$       \\ \hline
\end{tabular}
\vspace{.1cm}
\end{center}
\end{table}

\section{Conclusions}\label{sec:conl}
We propose a new approach for SPD preconditioning for symmetric indefinite systems, based on the idea of implicitly constructing 
approximations to the inverse of the system matrix absolute value. A multigrid example of such a preconditioner is presented, for a real-valued Helmholtz problem. 
Our experiments demonstrate that PMINRES with the new MG absolute value preconditioner leads to an efficient iterative scheme, which has modest memory requirements 
and outperforms traditional GMRES based methods if available memory is tight.

%This paper presents a new approach for constructing SPD preconditioners for symmetric 
%indefinite systems.
%%
%The suggested strategy, called AV preconditioning, is based on the idea
%of constructing operators that resemble the inverse of the matrix
%absolute value. An example of a MG AV preconditioner for a real-valued Helmholtz problem
%has been constructed. 
%%
%Our experiments demonstrate that PMINRES with the new MG preconditioner
%leads to an efficient iterative scheme, which has modest memory requirements and 
%outperforms traditional GMRES based methods if only limited amount of storage is available. 
%
%
%%on %restarted GMRES with the 
%%indefinite MG preconditioning.  

\section*{Acknowledgments}
The authors thank Michele Benzi, Yvan Notay, and Joe~Pasciak 
for their comments on the draft of the manuscript.
% \vskip12pt

%\newpage

%\bibliographystyle{plainnat}
\def\refname{
\end{document}